\documentclass{amsart}

\usepackage[centertags]{amsmath}
\usepackage{amssymb,amsthm}
\usepackage[all]{xy}
\usepackage{hyperref}
\usepackage{graphicx}

\title[On the map of B\"okstedt-Madsen \dots ]{On the map of B\"okstedt-Madsen from the cobordism category 
to $A$-theory}
\author{George Raptis}
\address{Institut f\"{u}r Mathematik, 
         Universit\"{a}t Osnabr\"{u}ck, 
Albrechtstrasse 28a,
         D-49069 Osnabr\"{u}ck, 
         Germany}
\email{graptis@uni-osnabrueck.de}
\author{Wolfgang Steimle}
\address{Universit\"at Bonn\\
               Mathematisches Institut\\
               Endenicher Allee~60,
               D-53115 Bonn, Germany}
\email{steimle@math.uni-bonn.de}


\DeclareMathOperator{\colim}{colim}

\DeclareMathOperator{\cyl}{cyl}
\DeclareMathOperator{\Diff}{Diff}
\DeclareMathOperator{\Emb}{Emb}

\DeclareMathOperator*{\hocolim}{hocolim}

\DeclareMathOperator*{\holim}{holim}

\DeclareMathOperator{\id}{id}

\DeclareMathOperator{\map}{map}
\DeclareMathOperator{\mor}{mor}

\DeclareMathOperator{\sing}{sing}
\DeclareMathOperator{\sintop}{sing^{top}}
\DeclareMathOperator{\simp}{simp}

\newcommand{\cob}{\mathcal{C}_{d}}
\newcommand{\cobn}{\mathcal{C}_{d,n}}
\newcommand{\cobd}{\mathcal{C}_{d, \partial}}
\newcommand{\cobdn}{\mathcal{C}_{d,\partial,n}}
\DeclareMathOperator{\coass}{\mathrm{\nabla}}
\DeclareMathOperator{\scan}{\mathrm{scan}}

\renewcommand{\sin}{\mathrm{sing}}
\newcommand{\ob}{\mathrm{ob}}
\newcommand{\mto}{\mathrm{MTO}}
\newcommand{\Th}{\mathrm{Th}}
\newcommand{\Gr}{\mathrm{Gr}}

\newcommand{\cat}{\mathbf{cat}}

\newcommand{\inv}{^{-1}}
\newcommand{\fibr}[2]{\begin{pmatrix} {#1} \\ \downarrow \\ {#2}\end{pmatrix}}

\newcommand{\RR}{\mathbb{R}}

\newcommand{\NN}{\mathbb{N}}

\newcommand{\calR}{\mathcal{R}}
\newcommand{\calU}{\mathcal{U}}
\newcommand{\calI}{\mathcal{I}}
\newcommand{\calC}{\mathcal{C}}
\renewcommand{\hom}{\underline{\mathrm{Hom}}}
\renewcommand{\cyl}{\mathrm{Cyl}}

\newcommand{\Rfd}{\mathcal{R}^{hf}}
\newcommand{\Rd}{\mathcal{R}^{\delta}}
\newcommand{\Rdd}{\overline{\mathcal{R}}^{\delta}}

\newcommand{\fRfd}{\mathcal{R}^{hf}_{\mathrm{fib}}}

\newcounter{commentcounter}

\begin{document}

\theoremstyle{plain}
\newtheorem{thm}{Theorem}
\newtheorem{cor}[thm]{Corollary}
\newtheorem{lem}[thm]{Lemma}
\newtheorem{prop}[thm]{Proposition}
\newtheorem{claim}[thm]{Claim}
\theoremstyle{definition}
\newtheorem{defn}[thm]{Definition}
\newtheorem{obs}[thm]{Observation}
\newtheorem{constr}[thm]{Construction}

\theoremstyle{remark}
\newtheorem{rem}[thm]{Remark}

\setcounter{secnumdepth}{2}
\numberwithin{thm}{subsection}

\SelectTips{eu}{10}
\renewcommand{\theenumi}{\roman{enumi}}
\renewcommand{\labelenumi}{\textup{(}\theenumi\textup{)}}
\renewcommand{\theequation}{\arabic{equation}}

\begin{abstract}
B\"okstedt and Madsen defined an infinite loop map from the embedded $d$-dimensional cobordism category of Galatius, 
Madsen, Tillmann and Weiss to the algebraic $K$-theory of $BO(d)$ in the sense of Waldhausen. 
The purpose of this paper is to establish two results in relation to this map. The first result is 
that it extends the universal parametrized $A$-theory Euler characteristic of smooth bundles with compact 
$d$-dimensional fibers, as defined by Dwyer, Weiss and Williams. The second result is that it actually 
factors through the canonical unit map $Q(BO(d)_+) \to A(BO(d))$. 
\end{abstract}

\maketitle


\section{Introduction}

The parametrized Euler characteristic was defined by Dwyer, Weiss and Williams in \cite{Dwyer-Weiss-Williams(2003)} for fibrations whose fibers are homotopy equivalent to a finite CW complex. Broadly speaking, the Euler characteristic of such a fibration $p: E \to B$ is a map that associates to every $b \in B$ the Euler class of the fiber $p \inv(b)$. The precise definition, which is given in terms of Waldhausen's algebraic $K$-theory of spaces ($A$-theory) \cite{Waldhausen(1985)}, produces this way a section of an associated fibration 
$$A_B(p): A_B(E) \to B$$ that is defined by applying the $A$-theory functor to $p$ fiberwise.
 
In the case where the fibration is actually a smooth fiber bundle and the fibers are compact smooth $d$-manifolds, possibly with boundary, the ``smooth Riemann-Roch theorem'' of \cite{Dwyer-Weiss-Williams(2003)} asserts that this fiberwise Euler characteristic can be identified with the composition of a stable transfer map, in the sense of Becker and Gottlieb \cite{BeGo}, followed by the unit transformation from stable homotopy to algebraic $K$-theory. More concretely, if we consider the vertical tangent bundle of the smooth fiber bundle $p: E \to B$ and pass to $BO(d)$, the parametrized $A$-theory Euler characteristic gives a map
\[\chi^{DWW} \colon B\to A(BO(d)),\]
and according to the smooth Riemann-Roch theorem, the diagram
\begin{equation} \label{smooth_RR_intro}
\xymatrix{
B \ar[rr]^{tr} \ar[rrd]_{\chi^{DWW}} && Q(BO(d)_+) \ar[d]^\eta\\
&& A(BO(d))   
}
\end{equation}
is commutative up to homotopy, where the map $tr$ is given by the classical Becker-Gottlieb transfer and $\eta$ 
denotes the unit map at $BO(d)$.

Let $\cob$ be the embedded $d$-dimensional cobordism category of \cite{GMTW(2009)}. Roughly speaking, the 
objects are closed smooth $(d-1)$-manifolds and the morphisms are cobordisms between them, all embedded in 
some high dimensional Euclidean space. Every closed smooth $d$-manifold $M$, embedded in some high dimensional Euclidean space, may be regarded as a cobordism from the empty manifold to itself and therefore it defines a loop in $B\cob$. This rule defines a map 
\[i_M: B\Diff(M)\to \Omega B\cob
\]
where $B\Diff(M)$ is the classifying space of smooth fiber bundles with fiber $M$. Recently, B\"okstedt and Madsen \cite{Boekstedt-Madsen(2011)} defined an infinite loop 
map
\[\tau\colon \Omega B\cob\to A(BO(d))\]
which, in non-technical language, is given by viewing an $n$-simplex in the nerve of $\cob$ as a filtered space equipped with a map to $BO(d)$ defined by the tangent bundle.  This raises naturally the following two questions: 
\begin{enumerate}
\item[(1)] Does the restriction of the map $\tau$ to $B\Diff(M)$ agree up to homotopy with the parametrized $A$-theory Euler characteristic of the universal bundle over $B\Diff(M)$?
\item[(2)]  Does the map $\tau$ also factor up to homotopy through stable homotopy, via the unit map $\eta$, as in the smooth Riemann-Roch theorem above?
\end{enumerate}
B\"okstedt and Madsen  \cite{Boekstedt-Madsen(2011)} expressed their belief that the answer to both questions is affirmative.

The purpose of this paper is to show that both statements are indeed true. The question of extending the universal parametrized $A$-theory Euler characterstic to the cobordism category can be regarded as a question about the 
additivity property of the parametrized $A$-theory Euler characteristic with respect to the fiber. Assuming that (1) is 
true, then Question (2) can also be regarded as a question about a structured additivity property of the factorization of
the universal parametrized $A$-theory Euler characteristic through the unit map as in Diagram \eqref{smooth_RR_intro}.
The first main ingredient in the proofs is to consider the cobordism category $\cobd$ of 
compact smooth manifolds with boundary, studied by Genauer \cite{Genauer(2008)}, which contains $\cob$ as a subcategory. The B\"okstedt-Madsen map can be extended to a map $$\tilde \tau\colon \Omega B\cobd\to A(BO(d)).$$ 

The space $\Omega B\cobd$ receives a map from $B\Diff(M)$, defined as before, for every $M$ compact smooth $d$-manifold, possibly with boundary.  In Theorem \ref{thm:main_result_1}, we show that the restriction of $\tilde\tau$ to $B\Diff(M)$ agrees up to homotopy with the composition of the universal parametrized $A$-theory Euler characteristic followed by the map to $A(BO(d))$ defined  by the vertical tangent bundle. The proof uses the second main ingredient, namely, that the universal bundle over $B\Diff(M)$ defines a \emph{bivariant A-theory characteristic} in the bivariant $A$-theory of the bundle (see \cite{Williams(2000)}), and that the universal parametrized $A$-theory Euler characteristic is the image of this characteristic under a \emph{coassembly map}. Since a basic problem in comparing all these maps is to find first the right identifications between the various models used to represent the various homotopy types, bivariant $A$-theory becomes extremely useful here, because it can offer a unifying perspective.

The homotopy type of $\Omega B\cobd$ was identified by Genauer \cite{Genauer(2008)} to be equivalent to $Q (BO(d)_+)$. To answer Question (2), we show in Theorem \ref{thm:main_result_2} that, under this identification, 
the map $\tilde\tau$ agrees with the unit map. This provides a geometric description of the unit map at $BO(d)$ in terms of smooth $d$-dimensional cobordisms. As consequence of this, the B\"okstedt-Madsen map $\tau$ factors up to homotopy as the following composition of a parametrized Pontryagin-Thom collapse map with the unit map:
$$\Omega B \cob \stackrel{\sim}{\to} \Omega^{\infty} \mto(d) \to Q(BO(d)_+) \to A(BO(d))$$
where the first map is the weak equivalence of \cite{GMTW(2009)} and the second map is defined by the canonical 
inclusion of Thom spectra. In particular, the homotopy commutativity of Diagram \eqref{smooth_RR_intro} is also a consequence of these two theorems.

\subsection*{Organization of the paper} In section 2, we recall the definitions of the cobordism categories $\cob$ and $\cobd$ and 
state the main results about their homotopy types from \cite{GMTW(2009)} and \cite{Genauer(2008)} respectively. In section 3 and 
appendix A, we discuss the bivariant $A$-theory of a fibration and study some of its properties. Only very special instances of 
bivariant $A$-theory will appear in the proofs of the main results, however we hope that the results here will also be of 
independent interest. In section 4, we review the construction of the $A$-theory coassembly map and recall the definition of the 
parametrized $A$-theory Euler characteristic from \cite{Dwyer-Weiss-Williams(2003)}, \cite{Williams(2000)}. In section 5, we prove 
the main results of the paper, answering Questions (1) and (2) above. Finally, in section 6, we end with a couple of remarks. First, we explain how our result 
generalize to cobordism categories with arbitrary tangential structures, and second, we comment on the connection with the work 
of Tillmann \cite{Tillmann(1999)} where a map analogous to the B\"okstedt-Madsen map was defined in the case of (a discrete 
version of) the oriented $2$-dimensional cobordism category. 

\subsection*{Acknowledgements}
The first author would like to thank the Mathematical Institute of the University of Bonn for its hospitality. The second 
author was supported by the Hausdorff-Center for Mathematics. He would also like to thank Johannes Ebert for helpful conversations.

\section{The cobordism categories $\cob$ and $\cobd$}

In this section we recall the main results about the homotopy types of the embedded $d$-dimensional cobordism categories $\cob$ and $\cobd$ from \cite{GMTW(2009)} and \cite{Genauer(2008)} respectively. 

\subsection{The cobordism category $\cob$.} For every 
$n \in \NN \cup \{\infty\}$, there is a topological category $\cobn$ defined as follows. An object of $\cobn$ is a pair $(M, a)$ where $a \in \RR$ and $M$ is a closed smooth $(d-1)$-dimensional submanifold of $\RR^{d-1 + n}$. (For $n= \infty$, 
define $\RR^{d-1 + \infty} := \colim_{n\to\infty} \RR^{d-1 + n}$ with the weak topology.) A non-identity morphism from $(M_0,a_0)$ to $(M_1, a_1)$ is a triple $(W, a_0,a_1)$ where $a_0 < a_1$ and $W$ is a compact smooth $d$-dimensional submanifold of $[a_0,a_1] \times \RR^{d-1+n}$ such that for some $\epsilon > 0$, we have:
\begin{itemize}
\item[(i)] $W \cap ([a_0, a_0 +\epsilon) \times \RR^{d-1 + n}) = [a_0, a_0 + \epsilon) \times M_0$
\item[(ii)] $W \cap ((a_1 - \epsilon, a_1] \times \RR^{d-1 + n}) = (a_1 - \epsilon, a_1] \times M_1$
\item[(iii)] $\partial W = W \cap (\{a_0, a_1\} \times \RR^{d-1+ n})$. 
\end{itemize}
Composition is defined by taking the union of subsets of $\RR \times \RR^{d-1 + n}$. The identities are formally added and regarded as ``thin'' product cobordisms. We abbreviate $$\cob :=  \mathcal{C}_{d, \infty} = \underset{n \to \infty}{\colim} \cobn.$$ 

The topology is defined as follows. For technical reasons, we work here with the slightly modified model discussed 
in \cite[Remarks 2.1(ii) and 4.5]{GMTW(2009)}. Set $$B_n(M) = \Emb(M, \RR^{d-1 + n}) / \Diff(M).$$ 
Let $\RR^{\delta}$ denote the set of real numbers with the discrete topology. The space of objects $\ob \cobn$ is 
$$\ob \cobn \cong \RR^{\delta} \times \underset{M}{\coprod} B_n(M)$$
where $M$ varies over the diffeomorphism classes of closed $(d-1)$-manifolds. By Whitney's embedding theorem, 
the space $\Emb(M, \RR^{d-1 + \infty})$ is contractible, and so there is a homotopy equivalence $B_{\infty}(M) \simeq B \Diff(M)$. 

The definition of the topology on the morphisms is similar, but requires in addition that the collars are preserved under
 the diffeomorphisms. In detail, given a cobordism $(W, h_0, h_1)$ from $M_0$ to $M_1$ with collars 
$h_0: [0,1) \times M_0 \to W$ and $h_1: (0,1] \times M_1 \to W$, and $0 < \epsilon < 1/2$, let 
$$\Emb_{\epsilon}(W, [0,1] \times \RR^{d-1 + n})$$ be the subspace of smooth embeddings that restrict to 
product embeddings on the $\epsilon$-neighborhood of the collared boundary (see \cite{GMTW(2009)} for a more precise definition).
This technical assumption is crucial in order to have a well-defined composition of morphisms. Set 
$$\Emb(W, [0,1] \times \RR^{d-1 + n}): = \underset{\epsilon \to 0}{\colim}{\Emb_{\epsilon}(W, [0,1] \times \RR^{d-1 + n})}.$$ 
Let $\Diff_{\epsilon}(W)$ denote the group of diffeomorphisms of $W$ that restrict to product diffeomorphisms on the 
$\epsilon$-neighborhood of the collared boundary. Set 
$$\Diff(W) = \Diff(W,h_0,h_1) : = \underset{\epsilon \to 0}{\colim}{\Diff_{\epsilon}(W)}.$$
There is a principal $\Diff(W)$-action on $\Emb(W, [0,1] \times \RR^{d-1 + n})$. Set 
$$B_n(W):=\Emb(W, [0,1] \times \RR^{d-1 + n}) / \Diff(W).$$ 
Then the space of morphisms $\mor \cobn$ is 
$$\mor \cobn \cong \ob \cob \sqcup \underset{W}{\coprod} ((\RR^{2}_+)^{\delta} \times B_n(W))$$
where $W=(W, h_0, h_1)$ varies over the diffeomorphism classes of $d$-dimensional cobordisms and $(\RR^{2}_+)^{\delta}$ 
denotes the open half plane $\{(a_0,a_1) \colon a_0 < a_1 \}$ with the discrete topology. We also have a homotopy
equivalence $B_{\infty}(W) \simeq B \Diff(W)$. 

We will be mainly interested in the ``stable'' case $n = \infty$. We recall the main result of \cite{GMTW(2009)} that identifies the homotopy 
type of the classifying space $B \cob$. Let $\Gr_d(\RR^{d+k})$ be the Grassmannian of $d$-dimensional linear subspaces in $\RR^{d+k}$ and consider the two standard bundles over it : the tautological $d$-dimensional vector bundle $\gamma_{d,k}$ and its $k$-dimensional complement $\gamma^{\perp}_{d,k}$. The spectrum $\mto (d)$ is the Thom spectrum associated to 
the inverse of the tautological vector bundle $\gamma_d :=\gamma_{d,\infty}$ over $\Gr_d(\RR^{d+ \infty})$, i.e.  
$$\mto (d)_{d+k} \colon = \Th (\gamma^{\perp}_{d,k})$$ and the structure maps are induced, after passing to Thom spaces, from the pullback diagrams, 
\[
\xymatrix{
\gamma^{\perp}_{d,k} \oplus \epsilon^1 \ar[d] \ar[r] & \gamma^{\perp}_{d,k+1} \ar[d] \\
\Gr_d(\RR^{d+k}) \ar[r] & \Gr_d(\RR^{d+k+1}).
}
\]

\begin{thm}[Galatius-Madsen-Tillmann-Weiss \cite{GMTW(2009)}] \label{GMTW}
There is a weak equivalence $$\alpha \colon B \cob \stackrel{\sim}{\longrightarrow} \Omega^{\infty -1} \mto (d).$$
\end{thm}

\subsection{The cobordism category $\cobd$.} Following similar methods, Genauer generalized the results of \cite{GMTW(2009)} to cobordism categories of manifolds with corners \cite{Genauer(2008)}. We will be mainly interested in the special case of manifolds with boundary. For every $n \in \NN \cup \{\infty\}$, there is a cobordism category $\cobdn$ of smooth $d$-dimensional cobordisms between manifolds with boundary, nicely embedded in $\RR \times \RR^{d-1+n}$. The precise definition is analogous:
\begin{itemize} 
\item[(i)$'$] an object is a pair $(M, a)$ where $a \in \RR^{\delta}$ and $M$ is a smooth neat $(d-1)$-dimensional
submanifold of $\RR_+ \times \RR^{d-2 + n}$. (This model of ``discrete cuts'' is not considered in \cite{Genauer(2008)}, however the same remarks as in \cite[Remarks 2.1(ii) and 4.5]{GMTW(2009)} apply in this case as well.) 
\item[(ii)$'$] A non-identity morphism from $(M_0,a_0)$ to $(M_1, a_1)$ is a triple $(W, a_0,a_1)$ where $a_0 < a_1$ and $W$ is a smooth neat $d$-dimensional submanifold (with corners) of $[a_0,a_1] \times \RR_+ \times \RR^{d-2+n}$ satisfying (i)-(iii) as above; composition of morphisms is by taking the union of subsets.
\item[(iii)$'$] The topology is defined similarly by the orbit spaces of the actions of diffeomorphisms on spaces of neat embeddings; see \cite{Genauer(2008)} for a precise definition.
\end{itemize}
We abbreviate $\cobd := \mathcal{C}_{d, \partial, \infty} = \underset{n \to \infty}{\colim} \cobdn$. 
 
\begin{thm}[Genauer \cite{Genauer(2008)}] \label{Genauer}
There is a weak equivalence $$\tilde{\alpha}: B \cobd \stackrel{\sim}{\longrightarrow} \Omega^{\infty -1}\Sigma^{\infty}BO(d)_+.$$

\end{thm}

Both weak equivalences are obtained as parametrized versions of the Pontryagin-Thom collapse map. We recall first 
the description of this collapse map in the case of a single compact, possibly with boundary, smooth $d$-manifold $M$ neatly embedded in 
$(0, 1) \times \RR_+ \times \RR^{d-2+n}$. This can be regarded as a(n) (endo)morphism of $\cobd$, essentially from the 
empty manifold to itself, and therefore it defines a loop in $B \cobd$. (To be precise, one should think of the empty manifold 
situated, say, inside $\{0\} \times \RR^{\infty}$ and $\{1\} \times \RR^{\infty}$ together with the canonical path in $B \cobd$ 
that connects these two points through the empty cobordism in $[0,1] \times \RR^{\infty}$.) 
Hence the image of this loop under the map $\Omega(\tilde\alpha)$ is a loop in $\Omega^{\infty -1} \Sigma^{\infty} BO(d)_+$. This can 
be roughly described as follows: consider the Pontryagin-Thom collapse map
\[ (S^{d-1+n} \wedge (\RR_+ \cup \{\infty\}), S^{d-1 + n} \times \{0\}) \to (\mathrm{Th}(\nu_M), \mathrm{Th}(\nu_{\partial M}))
\]
and the classifying map of the normal bundle 
\[ (\mathrm{Th}(\nu_M), \mathrm{Th}(\nu_{\partial M})) \to (\mto(d)_{d + n}, \mto(d-1)_{d -1 + n}). \]
The cofiber of the inclusion of spectra $\Sigma^{-1} \mto(d-1) \hookrightarrow \mto(d)$ is equivalent to the spectrum $\Sigma^{\infty} (BO(d)_+)$ \cite[Proposition 3.1]{GMTW(2009)}. So the composite map of pairs induces a 
stable map on cofibers,
\[ \Sigma^{\infty} S^0 \to \Sigma^{\infty} (BO(d)_+) \]
which essentially defines the image of $\tilde{\alpha}$ at the embedded manifold $M$. On the other hand, if $\partial M = 
\varnothing$, then the composite map is a loop in $\Omega^{\infty -1} \mto(d)$, 
\[ S^{d + n} \to \mto(d)_{d + N} \]
which essentially defines the image of $\alpha$ at the embedded closed manifold $M$. (This is not a precise definition because it depends
on various choices which are not canonical in $M \subseteq (0,1) \times \RR_+ \times \RR^{d-2+n}$, however, they are essentially unique 
in a homotopical sense.) 

More generally, in the parametrized case, there is an inclusion map 
$$i_M: B_{\infty}(M) \hookrightarrow \cobd((\varnothing,0),(\varnothing, 1)) \to  \Omega_{\varnothing} B \cobd$$ 
and the definition above of $\tilde{\alpha}$ at a point of $B_{\infty}(M)$ extends similarly to $B_{\infty}(M)$. For every $n \in \mathbb{N}$,
consider the following $M$-bundle together with its natural fiberwise neat embedding,
\[
\xymatrix{
\Emb(M, (0,1) \times \RR_+ \times \RR^{d-2 + n}) \times_{\Diff(M)} M \ar@{^{(}->}[r] \ar[d]^{\pi} & B_n(M) \times (0,1) \times \RR_+ \times \RR^{d-2 + n} \ar[dl] \\
B_n(M) 
}
\]
Let $\nu^{\pi}_M$ denote the fiberwise normal bundle of the embedding and $\nu^{\partial \pi}_{\partial M}$ the corresponding 
normal bundle associated to the $\partial M$-subbundle. The Pontryagin-Thom construction produces a collapse map
\[ 
(S^{d-1+n} \wedge (\RR_+ \cup \{\infty\}), S^{d-1 + n} \times \{0\}) \wedge B_n(M)_+  \to (\mathrm{Th}(\nu^{\pi}_M), \mathrm{Th}(\nu^{\partial \pi}_{\partial M}))
\]
and the classifying map of the normal bundle is a map 
\[ (\mathrm{Th}(\nu^{\pi}_M), \mathrm{Th}(\nu^{\partial \pi}_{\partial M})) \to (\mto(d)_{d + n}, \mto(d-1)_{d -1 + n}). \]
The composite map of pairs induces a stable map on cofibers,
\[ \Sigma^{\infty} (B_n(M)_+) \to \Sigma^{\infty} (BO(d)_+). \]
Letting $n \to \infty$, we obtain a map 
$$B_{\infty}(M) \to \Omega^{\infty} \Sigma^{\infty} BO(d)_+$$
which is up to homotopy the restriction of $\Omega(\tilde{\alpha})$ along the map $i_M$. Similarly, if $\partial M = \varnothing$, then we 
have the composite map 
$$\Sigma^{d + n} (B_n(M)_+) \to \mathrm{Th}(\nu^{\pi}_M) \to \mto(d)_{d + n}$$
and letting $n \to \infty$, we obtain a map $$B_{\infty}(M) \to \Omega^{\infty} \mto(d)$$
which is up to homotopy the restriction of $\Omega(\alpha)$ along the map $B_{\infty}(M) \to \Omega B \cob$. 

Note that there is an inclusion functor of cobordism categories $\cob \hookrightarrow \cobd$. The induced map on (the loop spaces of) the classifying spaces can be identified with the map of spectra $$\mto (d) \to \Sigma^{\infty} (BO(d)_+)$$
defined by the canonical inclusion of Thom spaces $\Th(\gamma^{\perp}_{d,k}) \hookrightarrow \Th(\gamma^{\perp}_{d,k} \oplus \gamma_{d,k}) \cong S^{d+k} \wedge BO(d)_+$. We refer the reader to \cite[Section 6]{Genauer(2008)} for more details.

\section{Bivariant $A$-theory}

Bivariant $A$-theory was defined by Bruce Williams \cite{Williams(2000)}. A less general ``untwisted'' version can be discovered in unpublished work of Waldhausen.  A variation of the latter was also considered by B\"okstedt and Madsen \cite{Boekstedt-Madsen(2011)}.

The purpose of this section is to review and, for technical convenience, slightly modify Williams's definition of  bivariant $A$-theory. This associates to a fibration $p\colon E\to B$ a \emph{bivariant $A$-theory spectrum} $A(p)$ that has the following properties:
\begin{itemize}
\item[(a)] If $B$ is the one-point space, then $A(p) = A(E)$.
\item[(b)] For every fibration $q\colon V \to B$ and fiberwise map $f\colon E\to V$ over $B$, there is a natural push-forward map $f_*\colon A(p)\to A(q)$. Moreover, push-forward maps are homotopy invariant, i.e. if $f$ is a
homotopy equivalence, then so is $f_*$.
\item[(c)] For every pullback square
\[\xymatrix{
E \times_B B'  \ar[rr] \ar[d]^{p'} && E \ar[d]^p\\
B' \ar[rr]^g && B
}\]
there is a natural pull-back map $g^*\colon A(p)\to A(p')$. Moreover, pull-back maps are homotopy invariant, i.e. if $g: B' \to B$ is a homotopy equivalence, then so is $g^*$.
\item[(d)]  Push-forward maps commute with pull-back maps, i.e. given maps $q$, $f$ and $g$ as above, the following diagram commutes
\[\xymatrix{
A(p) \ar[rr]^{f_*} \ar[d]^{g^*} && A(q) \ar[d]^{g^*} \\
A(p') \ar[rr]^{f'_*} && A(q')
}\]
where $q'$ is the pullback of $q$ along $g$ and $f' \colon E \times_B B' \to V \times_B B'$ is the map induced by $f$.
\item[(e)] For every composable pair of fibrations $E \stackrel{q}{\to} V \stackrel{q'}{\to} B$ where $p = q' \circ q$, there is a
product map $$A(q) \wedge A(q') \to A(p)$$ 
which is natural up to canonical homotopy.
\end{itemize}

\subsection{Definition of bivariant $A$-theory.} The space $A(p)$ is the $K$-theory of a Waldhausen category of retractive spaces over $E$ that are suitably related to 
the fibration $p$. As usual, we assume that all spaces are compactly generated and Hausdorff. For technical reasons,
we also make the following assumption \emph{throughout this section}.

\

\noindent \textbf{Assumption.} The base space $B$ of the fibration $p: E \to B$ has the homotopy type of a CW complex. (But see 
also Remark \ref{CW-condition}.)

\

The category $\calR(E)$ of retractive spaces over $E$ consists of all diagrams of spaces
\[E \overset{i}{\rightarrowtail} X \xrightarrow{r} E\]
where $r\circ i=\id_E$ and $i$ is a cofibration.  A morphism of retractive spaces is a map over and under $E$. 
The category $\calR(E)$ becomes a Waldhausen category if we define cofibrations (resp.~weak equivalences) to be
those morphisms whose underlying map of spaces is a cofibration (resp. homotopy equivalence). Let $\Rfd(E)\subset \calR(E)$ 
be the full subcategory of all objects $(X,i,r)$
which are homotopy finite, i.e.~which are weakly equivalent, in $\calR(E)$, 
to an object $(X', i', r')$ such that $(X', i'(E))$ is a finite relative CW-complex. This is a Waldhausen subcategory
of $\calR(E)$ whose $K$-theory, denoted by $A(E)$, is the algebraic $K$-theory of the space $E$ 
\cite{Waldhausen(1985)}. 

For the definition of the bivariant $A$-theory of $p$, we consider those retractive spaces 
over $E$ that define families of homotopy finite retractive spaces over the fibers of $p$, 
parametrized by the points of $B$.

\begin{defn} \label{bivariant category}
Let $p\colon E\to B$ be a fibration. The category $\Rfd(p)\subset \calR(E)$ is the full subcategory of 
all objects $E \overset{i}{\rightarrowtail} X \xrightarrow{r} E$ such that:
\begin{enumerate}
\item the composite $p \circ r$ is a fibration, and
\item for each $b\in B$, the space $(p \circ r)\inv(b)$ is homotopy finite as an object of $\calR(p\inv(b))$ 
(with the obvious structure maps).
\end{enumerate}
\end{defn} 

From our general assumption on $B$, it follows that for every object $(X, i,r)$ of $\Rfd(p)$, the pair 
$(X,i(E))$ is homotopy equivalent to a relative CW-complex. (This is a special case of Lemma \ref{CW-homotopy-type}.) 
We define a cofibration, resp.~weak equivalence, in $\Rfd(p)$ to be a morphism which is a cofibration, 
resp.~weak equivalence, in $\calR(E)$. 

\begin{prop}
The category $\Rfd(p)$ is a Waldhausen subcategory of $\calR(E)$. Moreover, it satisfies the ``2-out-of-3" axiom (i.e. it is saturated in the terminology of \cite{Waldhausen(1985)}) 
and admits functorial factorizations of morphisms into a cofibration followed by a weak equivalence.
\end{prop}
\begin{proof}
Since $\Rfd(p)\subset\calR(E)$ is a full subcategory which contains the zero object, it suffices 
to show that $\Rfd(p)$ is closed under pushouts along a cofibration in $\calR(E)$.  Let
\[\xymatrix{
X_0 \ar[rr] \ar@{>->}[d] && X_1 \ar@{>->}[d]\\
X_2 \ar[rr] && X
}\]
be a pushout diagram of retractive spaces over $E$, such that $p\circ r_i\colon X_i\to B$ are fibrations, for $i=0,1,2$, whose 
fibers are homotopy finite relative to the fibers of $p$. Then the induced map $p \circ r \colon X\to B$ is a fibration 
(see \cite[p. 383]{Kieboom}), and there is a pushout diagram
\[\xymatrix{
(p \circ r_0)\inv(b) \ar[rr] \ar@{>->}[d] && (p \circ r_1)\inv(b) \ar@{>->}[d]\\
(p \circ r_2)\inv(b) \ar[rr] && (p \circ r)\inv(b)
}\]
which shows that $(p \circ r)\inv(b)$ defines an object of $\Rfd(p\inv (b))$, since this category is closed under taking 
such pushouts. The class of homotopy equivalences clearly satisfies the ``2-out-of-3" axiom, so $\Rfd(p)$ is saturated. 
It remains to show the existence of factorizations of morphisms. These will be obtained by 
the mapping cylinder construction as usual. Let $f: (X,i_X, r_X) \to (Y, i_Y, r_Y)$ be a morphism in $\Rfd(p)$. Consider 
$$(X \times I, j_0 \circ i_X, r_X \circ \pi_X)$$
as an object of $\Rfd(E)$, where $j_0(x) = (x,0)$ and $\pi_X(x,t) = x$. A cylinder object $\cyl_E(X)$ for $(X, i_X, r_X)$ is defined by the pushout square in $\Rfd(p)$: 
\[
\xymatrix{
E \times I \ar[rr]^{\mathrm{proj}} \ar@{>->}[d]^{i_X \times \id}  && E \ar@{>->}[d] \\
X \times I \ar[rr]^q && \cyl_E(X).
}
\]
Then the mapping cylinder $M_f$ of the map $f: (X,i_X, r_X) \to (Y, i_Y, r_Y)$ is defined by the pushout in $\Rfd(p)$:
\[
\xymatrix{
X \ar[rr]^f \ar@{>->}[d]^{q \circ  j_0} && Y \ar@{>->}[d] \\
\cyl_E(X) \ar[rr]^u && M_f
}
\]
and is denoted by $(M_f, i', r')$. Note that the fiber of $p \circ r': M_f \to B$ at $b \in B$ fits in the pushout diagram 
 \[\xymatrix{
(p \circ r_X)\inv(b) \ar[rr] \ar@{>->}[d]^{q \circ j_0} && (p \circ r_Y)\inv(b) \ar@{>->}[d]\\
\cyl_{p \inv(b)}((p \circ r_X)\inv(b)) \ar[rr] && (p \circ r')\inv(b)
}\]
By the universal property of pushouts, there is a canonical map $(M_f, i', r') \to (Y, i_Y, r_Y)$ which is also a homotopy equivalence. Then the standard factorization of the map $f: (X, i_X, r_X) \to (Y, i_Y, r_Y)$ as 
\[
\xymatrix{
(X, i_X, r_X) \ar@{>->}[rr]^{u \circ q \circ j_1} && (M_f, i', r') \ar[r] & (Y, i_Y, r_Y)
}
\]
defines functorial factorizations in $\Rfd(p)$ with the required properties.
\end{proof}

\begin{rem} If $p\colon X\times B\to B$ is the trivial fibration, 
then the Waldhausen category $\Rfd(p)$ is closely related to the bivariant category denoted by $W(X,B)$ in 
\cite{Boekstedt-Madsen(2011)}. Later on (subsection \ref{scanning_map}), this notation will be used to denote 
the (classifying space of the) weak equivalences of $\Rfd(p)$. From now on, when we discuss the homotopy type of a small category, we will often omit the classifying space functor ``$B$'', or simply replace it by ``$\vert \cdot \vert$'', in order to simplify the notation.  
\end{rem} 

\begin{defn}
The bivariant $A$-theory of $p : E \to B$ is defined to be the space 
\begin{center}
$A(p) :=K(\Rfd(p))=\Omega\vert wS_\bullet\Rfd(p)\vert$.
\end{center}
\end{defn}

Most of this section is devoted to the proof of the properties of bivariant $A$-theory which were stated at the beginning. First, notice that if $B$ is a point, then the categories $\Rfd(p)$ and $\Rfd(E)$ are the same, so we have $A(p)=A(E)$ in this case. This shows property (a).

\subsection{Functoriality.} We now proceed to define the push-forward and pull-back maps. Let $q\colon V \to B$ be another fibration and $f\colon E\to V$ a fiberwise map, i.e.~$q\circ f=p$. The push-forward along $f$ defines an exact functor of Waldhausen categories
\[f_*\colon \calR(E)\to\calR(V), X\mapsto X\cup_E V.\]
We claim that this actually restricts to an exact functor 
\[f_*\colon \Rfd(p)\to\Rfd(q)\]
between the corresponding Waldhausen subcategories. Indeed we have already remarked that if $X$, $E$, and $V$ are fibered over $B$, then so is also the adjunction space $X\cup_E V$. Moreover, the fiber of $X \cup_E V$ over a 
point $b\in B$ is the adjunction space $X_b\cup_{E_b} V_b$ and it is homotopy finite relative $V_b$ whenever $X_b$ is homotopy finite relative
$E_b$. Hence we obtain a map in $K$-theory,
\[f_*\colon A(p)\to A(q).\]

To define the pull-back maps, consider a pullback square 
\begin{equation}\label{eq:pull_back_square}
\xymatrix{
E' \ar[d]^{p'} \ar[rr] && E\ar[d]^p\\
B' \ar[rr]^g && B
}
\end{equation}
There is a functor
\[g^*\colon \Rfd(p)\to\Rfd(p')\]
defined by sending a retractive space $X$ over $E$ to the pullback $X':=X\times_B B'$. This defines a retractive space over $E'$ 
and a fibration over $B'$. Also for each $b'\in B'$ the fiber $X'_{b'}\cong X_{g(b)}$ is homotopy finite as a retractive space over $E'_{b'}\cong E_{g(b)}$. This shows that the functor is well-defined. Moreover, it preserves pushouts, cofibrations 
(see \cite[p. 381]{Kieboom})
and homotopy equivalences, so it defines an exact functor of Waldhausen categories. Hence we obtain a map in $K$-theory,
\[g^*\colon A(p)\to A(p').\]

\begin{rem}[Naturality]
In order to obtain strict naturality of these maps (and also to ensure that the size of the Waldhausen categories is small) we have 
to make certain additional assumptions. Fix, once and for all, a set $\calU$ of cardinality $2^{\vert\RR\vert}$. In the 
definition of an object $(X,i,r)$ in $\Rfd(p)$, where $p\colon E\to B$, we additionally require that $X$ is a set-theoretical 
subset of $E\amalg (B\times \calU)$, such that 
\begin{enumerate}
\item the composite
\[E\xrightarrow{i} X \hookrightarrow E\amalg (B\times \calU)\]
is the inclusion of $E$ into the disjoint union, and
\item the following diagram is commutative:
\[\xymatrix{
X \ar@{^{(}->}[rr] \ar[rd]_{p\circ r}
   && E\amalg (B\times \calU) \ar[ld]^{p\amalg \operatorname{proj}}\\  
& B
}\]
\end{enumerate}
For a map $f\colon E\to V$ over $B$, the adjunction space $X\cup_E V$ can be regarded as a subset of 
$V \amalg (B\times \calU)$ satisfying conditions (i) and (ii). On the other hand, suppose that we are given a pullback square 
\eqref{eq:pull_back_square}, then the pullback $X\times_B B'$ can be regarded as a subset of 
$E'\amalg (B'\times \calU)$. Using these conventions, both push-forward and pull-back maps are \emph{strictly functorial} and 
\emph{commute with each other}. This shows parts of properties (b) and (c) and property (d).
\end{rem}

\subsection{Homotopy invariance.} The following propositions show the homotopy invariance of bivariant $A$-theory.

\begin{prop}\label{lem:homotopy_invariance_of_bivariant_stuff2} Let $p: E \to B$ and $q: V \to B$ be fibrations and $f: E \to V$ a fiberwise map over $B$. If $f$ is a homotopy equivalence, then so are the induced push-forward maps 
$wS_nf_* : wS_n \Rfd(p) \to wS_n \Rfd(q)$ for all $n \geq 0$. In particular, the push-forward map $f_*: A(p) \to A(q)$ is also a homotopy equivalence.
\end{prop}
\begin{proof}
We show this first in the case where $f \colon E \stackrel{\simeq}{\hookrightarrow} V$ is a trivial cofibration by 
applying Cisinski's generalized approximation theorem \cite{Cisinski(2010)} (cf.~\cite[Theorem 1.6.7]{Waldhausen(1985)}). 
So it suffices to check that the exact functor $f_*: \Rfd(p) \to \Rfd(q)$ has the approximation properties (AP1) and 
(AP2) of \cite[p. 512]{Cisinski(2010)}. Indeed the approximation theorem of \cite[Proposition 2.14]{Cisinski(2010)} shows 
then that $wS_nf_*$ is a homotopy equivalence for all $n \geq 0$ (see \cite[Proposition 2.3, Lemme 2.13]{Cisinski(2010)}). 

Since $f$ is a homotopy equivalence, then clearly $g: X \to Y$ (over $E$) is a homotopy equivalence if and only if
$f_*(g): X \cup_E V \to Y \cup_E V$ is a homotopy equivalence, so (AP1) holds. Let $(X, i_X, r_X)$ be an object of $\Rfd(p)$, $(Y, i_Y, r_Y)$ an object of $\Rfd(q)$ and $$u: f_*(X,i_X, r_X) = (X \cup_E V, i'_X, r'_X) \to (Y, i_Y, r_Y)$$ a map in $\Rfd(q)$. We factorize the retraction map $r_Y$ into a trivial cofibration and a fibration 
$$Y \stackrel{j}{\rightarrowtail} Y' \stackrel{q}{\rightarrow} V.$$
Clearly $(Y', i_{Y'} = j \circ i_Y, q)$ is an object of $\Rfd(q)$ and its restriction $(Y'_{|E}, i_{Y'}, q)$ over $E$ is an object of $\Rfd(p)$. There is an adjoint map 
$$v: (X, i_X, r_X) \to (Y'_{|E}, i_{Y'}, q)$$ 
in $\Rfd(p)$. Then we have a diagram in $\Rfd(q)$ as follows 
\[
\xymatrix{
(X \cup_E V, i'_X, r'_X) \ar[d]^{f_*(v)} \ar[rr]^u && (Y, i_Y, r_Y) \ar[d]^{j}_{\simeq} \\
(Y'_{|E} \cup_E V, i_{Y'}, q) \ar@{>->}[rr]^{\simeq} && (Y', i_{Y'}, q)  
}
\]
and therefore (AP2) also holds. This concludes the proof in the case where $f$ is a trivial cofibration. The general case of an 
arbitrary homotopy equivalence $f: E \stackrel{\simeq}{\to} V$ follows from this by factorizing $f$ in the standard way as 
\[
\xymatrix{
E \ar@{>->}[rr]^(.4){\simeq} \ar[rrd]_{p} && (E \times I) \cup_E V \ar[d] \ar@<1ex>[rr] && V \ar@{>->}[ll]^(.4){\simeq} \ar[lld]^q \\
&& B &&
}
\]
to reduce this general case to the case of trivial cofibrations.
\end{proof}

\begin{cor} \label{htpy-inv: fiberwise}
Let $p: E \to B$ and $q: V \to B$ be fibrations and $f,g: E \to V$ two fiberwise maps over $B$. If $f \simeq_B g$ are fiberwise
 homotopic over $B$, then $wS_n f_* \simeq wS_n g_*: wS_n \Rfd(p) \to wS_n \Rfd(q)$ are homotopic for all $n \geq 0$. Moreover, $f_* \simeq g_* : A(p) \to A(q)$ are also homotopic.
\end{cor}
\begin{proof}
It suffices to prove the statement for the inclusions at the endpoints $$j_0, j_1\colon E\to E\times I$$
 regarded as fiberwise maps from $p$ to the fibration $q = p \circ \mathrm{proj} \colon E\times I \to B$. Both are split by the projection $\pi: E \times I \to E$ over $B$.  By Proposition \ref{lem:homotopy_invariance_of_bivariant_stuff2}, the push-forward maps $wS_n(j_0)_*$ and $wS_n (j_1)_*$ are homotopy equivalences with homotopy inverse given by $wS_n \pi_*$. It follows that they are homotopic. The last statement can be shown similarly.  
\end{proof}

\begin{prop}\label{lem:homotopy_invariance_of_bivariant_stuff}
Let $p: E \to B$ be a fibration and $g \colon B'\to B$ a map as in diagram \eqref{eq:pull_back_square}. If $g$ is a homotopy equivalence, then so are the induced pull-back maps $wS_ng^*: wS_n \Rfd(p) \to wS_n \Rfd(p')$ for all $n \geq 0$. In particular, the pull-back map $g^*: A(p) \to A(p')$ is also a homotopy equivalence. 
\end{prop}
\begin{proof}
It is enough to show that if $i_0, i_1\colon B\to B\times I$ are the inclusions at the endpoints, then the induced maps 
\[ i_0^*, i_1^* \colon  w\Rfd(p\times\id_I) \to  w\Rfd(p)\] 
are homotopic.  By Corollary \ref{htpy-inv: fiberwise}, it suffices to show that the maps 
\[(j_0)_*\circ i_0^*, (j_1)_*\circ i_1^*\colon  w\Rfd(p\times\id_I) \to  w\Rfd(q) \]
are homotopic. We recall that $j_0, j_1: E \to E \times I$ denote the inclusions at the endpoints, as fiberwise maps over $B$, and 
$q \colon E \times I \to E \stackrel{p}{\to} B$ is the composite fibration. Let
\[\pi\colon w\Rfd(p\times\id_I)\to w\Rfd(q)\]
be the forgetful functor which views a fibration over $B\times I$ as one over $B$. Then there are natural weak equivalences of functors
\[(j_0)_*\circ i_0^* \xrightarrow{\simeq} \pi\xleftarrow{\simeq} (j_1)_*\circ i_1^*\]
which give the desired homotopy after geometric realization. Applying the same argument in each degree of the $S_\bullet$-construction finishes the proof.
\end{proof}

The above statements conclude the proof of properties (b) and (c). As a consequence of the homotopy invariance, we can define a 
\emph{thick model} for $A$-theory as follows (see also \cite{Boekstedt-Madsen(2011)}). 
This model will be needed in the proofs of the main results. We abbreviate 
\[ \Rfd (X, B) :=  \Rfd \fibr{ X \times B}{B}.  \] 
The thick model for $|wS_q\Rfd(X)|$ is defined to be the geometric realization of the simplicial space 
\[
wS_q \Rfd(X,\Delta^\bullet):= \left[ [n] \mapsto \vert wS_q \Rfd(X, \Delta^n) \vert \right] 
\]
where $\Delta^n = | \Delta^n_{\bullet}|$ denotes the standard topological $n$-simplex and the simplicial operations are induced by 
the pull-back maps. The thick model for $A$-theory is defined to be the space
$$A_{\Delta}(X) : = \Omega | ([q], [n]) \mapsto wS_q \Rfd (X, \Delta^n) |$$
where 
$$([q], [n]) \mapsto wS_q \Rfd(X, \Delta^n)$$ 
is viewed as a bisimplicial space. By Proposition \ref{lem:homotopy_invariance_of_bivariant_stuff}, the 
inclusion of the 0-skeleton 
\[
wS_q \Rfd(X) \stackrel{\simeq}{\to} \left\vert w S_q \Rfd (X,\Delta^\bullet) \right\vert.
\] 
is a homotopy equivalence. Thus the bisimplicial space defining the thick model for $A$-theory is homotopically constant
in the $n$-direction. Passing to the loop spaces of the geometric realizations, we obtain a homotopy equivalence
\[ 
A(X) \stackrel{\simeq}{\longrightarrow} A_{\Delta}(X). 
\]

The proof of property (e), which will not be needed for the main results of this paper, will be discussed separately in 
appendix A. We note that, based on these properties, Fulton and MacPherson \cite{Fulton-MacPherson(1981)} presented an axiomatic 
approach to \emph{bivariant theories} and studied their connection with Riemann-Roch theorems (see also \cite{Williams(2000)}).

\begin{rem}\label{CW-condition}
The results of this section remain true without any special assumption on $B$. Our assumption is related to the choice between 
homotopy equivalences and weak homotopy equivalences. The homotopy finiteness condition of Definition \ref{bivariant category} 
does not imply in general that the objects of $\Rfd(p)$ are homotopy equivalent to relative CW-complexes. Thus, for a general 
fibration $p: E \to B$, it would be more reasonable to define $A(p)$ to be the space $A(\tilde{p})$ where 
$\tilde{p}: \tilde{E} \to \tilde{B}$ is the pullback of $p$ by a functorial CW-approximation 
$g: \tilde{B} \stackrel{\sim^w}{\longrightarrow} B$. Alternatively, the choice of weak homotopy equivalences as weak 
equivalences leads to a homotopy equivalent $K$-theory space. 
\end{rem}

\subsection{A model for the unit transformation.} \label{unit_map_model} 
We write $\mathbf{A}(X)$ and $\mathbf{K}(\mathcal{C})$, where $\mathcal{C}$ is a Waldhausen category, to denote the $\Omega$-spectrum defined by $A(X)$ and $K(\mathcal{C})$ respectively, obtained by iterating the $S_\bullet$-construction (see \cite{Waldhausen(1985)}). The unit transformation is a natural transformation of spectra 
$$\eta_X \colon \Sigma^{\infty} X_+ \longrightarrow \mathbf{A}(X).$$
For $X = \ast$, this is the map of spectra $\eta_{\ast} : \Sigma^{\infty} S^0 \to \mathbf{A}(\ast)$ which sends the non-basepoint of 
$S^0$ to the point $[S^0] \in A(\ast)$ corresponding to the based space $S^0$ as an object of $\Rfd(\ast)$. For general $X$,
$\eta_X$ is defined to be the composition
$$\Sigma^{\infty} X_+ \cong \Sigma^{\infty} S^0 \wedge X_+ \stackrel{\eta_{\ast} \wedge \mathrm{id}}{\longrightarrow} \mathbf{A}(\ast) \wedge X_+
\to \mathbf{A}(X)$$
where the last map is the assembly transformation for $A$-theory (see e.g. \cite{Dwyer-Weiss-Williams(2003)} for more details). 
For a geometric definition, following Waldhausen's manifold approach \cite{Waldhausen(1982)}, see also \cite{BDW(2009)}.

The purpose of this subsection is to define another model for the unit transformation. Let $\Rd(X)$ be the Waldhausen subcategory
of $\Rfd(X)$ with objects $(X \coprod S \leftrightarrows X)$ where $S$ is a discrete space. Note that weak equivalences in 
$\Rd(X)$ are isomorphisms and cofibrations are split. For technical reasons, we also consider a reduced version $\Rdd(X)$ of 
$\Rd(X)$, which is the full subcategory of $\Rd(X)$ containing the zero object and the objects: 
$$(X \coprod \{1, \cdots, m \} \leftrightarrows X).$$
Note that the inclusion $\Rdd(X) \to \Rd(X)$ is an equivalence of categories, so it induces a homotopy equivalence in $K$-theory. 
The category $\Rdd(X)$ does not detect the topology of $X$, i.e. $\Rdd(X)$ is isomorphic to $\Rdd(X^{\delta})$. We recall that $X^{\delta}$ denotes the space $X$ with the discrete topology. Moreover,
it is easy to see that 
$$\vert w \Rdd(X) \vert = \coprod_{m \geq 0} E\Sigma_m \times_{\Sigma_m} (X^{\delta})^m.$$
Since the cofibrations in $\Rdd(X)$ split, it follows that the canonical map
$$ \vert w \Rdd(X) \vert \to K(\Rdd(X))$$
is a group completion (see \cite[1.8]{Waldhausen(1985)}). By well-known results in the theory of infinite loop spaces (see e.g. \cite{Segal(1974)}), there is a natural stable equivalence
$$\Sigma^{\infty} X^{\delta}_+ \stackrel{\sim}{\to} \mathbf{K}(\Rdd(X))$$
which is defined by sending an element $x \in X^{\delta}$ to the associated retractive space $X \coprod \{1\} \leftrightarrows X$.
Also, following the methods of \cite{Boardman-Vogt}, \cite{May}, \cite{Segal(1973)}, one can also describe this equivalence 
geometrically by a natural (zigzag of) weak equivalence(s) of inifinite loop spaces
$$K(\Rdd(X)) \stackrel{\sim}{\longrightarrow} Q(X^{\delta}_+).$$ 

We can also define a bivariant version of $\Rd(X)$ as follows. Let $\Rd(X, \Delta^n)$ be the Waldhausen subcategory of
$\Rfd(X, \Delta^n)$ with objects: 
\[
\xymatrix{
Y \ar@<1ex>[r]^(.45)r \ar[dr]_q & X \times \Delta^n \ar[l]^(.55)i \ar[d] \\
& \Delta^n
}
\]
in $\Rfd(X, \Delta^n)$ such that in addition: for every $b\in \Delta^n$, the retractive space $((q)\inv(b), X)$ is an object
of $\Rd(X)$. Weak equivalences in $\Rd(X, \Delta^n)$ are isomorphisms and cofibrations are split. Similarly, we consider a reduced 
version $\Rdd(X, \Delta^n)$ of $\Rd(X, \Delta^n)$ which is the full subcategory with objects the zero object and the objects:
\[
\xymatrix{
(X \coprod \{1, \cdots, m\}) \times \Delta^n \ar@<1ex>[r] \ar[dr] & X \times \Delta^n \ar[l] \ar[d] \\
& \Delta^n
}
\]
The inclusion $\Rdd(X, \Delta^n) \to \Rd(X, \Delta^n)$ is an equivalence of categories, so it induces a homotopy 
equivalence in $K$-theory. Let $\sing_n(X) = \mathrm{Hom}(\Delta^n, X)$ denote the set of singular $n$-simplices of $X$. Then
observe that there is an isomorphism of categories 
$$\Rdd(X, \Delta^n) \cong \Rdd(\sin_n X)$$
and so we have
$$\vert w \Rdd(X, \Delta^n) \vert = \coprod_{m \geq 0} E \Sigma_m \times_{\Sigma_m} (\sing_n(X))^m.$$

We define the \emph{thick bivariant model} for the stable homotopy of $X$ to be the space
$$Q_{\Delta}(X) : = \Omega \vert ([q], [n]) \mapsto w S_q \Rd(X, \Delta^n) \vert$$
and its reduced version to be the space
$$\overline{Q}_{\Delta}(X) : = \Omega \vert ([q], [n]) \mapsto w S_q \Rdd(X, \Delta^n) \vert.$$
Note that the inclusion $\overline{Q}_{\Delta}(X) \stackrel{\sim}{\to} Q_{\Delta}(X)$ is a weak equivalence. We write 
$\mathbf{Q}_{\Delta}(X)$ and $\overline{\mathbf{Q}}_{\Delta}(X)$ to denote the associated $\Omega$-spectra. The terminology is justified by the following proposition.

\begin{prop} \label{thick_model_stable_htpy}
There is a natural stable equivalence 
$$\theta_X \colon \Sigma^{\infty} X_+ \simeq \Sigma^{\infty} |\sin_{\bullet} X|_+ \stackrel{\sim}{\longrightarrow} \overline{\mathbf{Q}}_{\Delta}(X) 
\simeq \mathbf{Q}_{\Delta}(X).$$
\end{prop}
\begin{proof}
We have the following identifications
\begin{multline*}
\vert ([q],[n]) \mapsto w S_q \Rdd(X, \Delta^n) \vert \cong \vert [n] \mapsto \vert [q] \mapsto w S_q \Rdd(X, \Delta^n) \vert \vert \cong
\\ \cong \vert [n] \mapsto B(\coprod_{m \geq 0} E\Sigma_m \times_{\Sigma_m} (\sing_n(X))^m) \vert  \cong B( \coprod_{m \geq 0} E\Sigma_m \times_{\Sigma_m} 
\vert \sin_{\bullet} X \vert^m)
\end{multline*}
where $B(-)$ is the classifying space of a topological monoid. Then there is a natural stable equivalence as required, which 
is defined by the inclusion 
$$|\sin_{\bullet}X| \hookrightarrow \coprod_{m \geq 0} E\Sigma_m \times_{\Sigma_m} \vert \sin_{\bullet}X \vert^m \to \Omega B( \coprod_{m \geq 0} E\Sigma_m \times_{\Sigma_m} \vert \sin_{\bullet}X \vert^m).$$
\end{proof}

The exact inclusions $\Rdd(X, \Delta^n) \hookrightarrow \Rd(X, \Delta^n) \hookrightarrow \Rfd(X, \Delta^n)$ induce maps 
between the $K$-theory spectra, and so also a natural map (of spectra) between the thick models:
$$\eta^{\Delta}_X : \overline{\mathbf{Q}}_{\Delta}(X) \stackrel{\sim}{\to} \mathbf{Q}_{\Delta}(X) \to \mathbf{A}_{\Delta}(X).$$

\begin{prop} \label{unit_map}
The following diagram of spectra commutes up to homotopy,
\[
\xymatrix{
\Sigma^{\infty} X_+ \ar[dr]^{\eta_X} & \Sigma^{\infty} |\sin_{\bullet} X|_+ \ar[r]^(0.6){\theta_X} \ar[l]_(0.6){\sim} &\overline{\mathbf{Q}}_{\Delta}(X) \ar[d]^{\eta^{\Delta}_X} \\  
& \mathbf{A}(X) \ar[r]^{\sim} & \mathbf{A}_{\Delta}(X). & 
}
\]
\end{prop}
\begin{proof}
Note that both compositions are natural transformations between spectra-valued functors from a functor that is excisive, i.e. it preserves homotopy pushouts. It follows that both compositions are determined by their evaluation at $X = \ast$ (see also \cite{WeWi(1995)}). Hence it suffices to show 
that the following diagram commutes up to homotopy, 
\[
\xymatrix{
\Sigma^{\infty} S^0 \ar[dr]^{\eta_\ast} & \Sigma^{\infty} |\sin_{\bullet} (\ast)|_+ \ar[r]^(0.6){\theta_\ast} \ar[l]_(0.6){=} &\overline{\mathbf{Q}}_{\Delta}(\ast) \ar[d]^{\eta^{\Delta}_\ast} \\  
& \mathbf{A}(\ast) \ar[r]^{\sim} & \mathbf{A}_{\Delta}(\ast). & 
}
\]
Then the result follows because both compositions are defined by the map
$$S^0 \to A_{\Delta}(\ast),$$
which sends the non-basepoint to the element of $A_{\Delta}(\ast)$ defined by $S^0$ as an object of $\Rfd(\ast)$. 
\end{proof}

\section{The parametrized $A$-theory Euler characteristic}

The purpose of this section is to review a description of the parametrized $A$-theory Euler characteristic of Dwyer, Weiss 
and Williams \cite{Dwyer-Weiss-Williams(2003)} using bivariant $A$-theory. Let $p\colon E\to B$ be a fibration with homotopy finite fibers. 
The retractive space $E\times S^0$ over $E$ is an object of $\Rfd(p)$, so it defines a point \[\chi(p)\in A(p)\] 
called the \emph{bivariant $A$-theory characteristic} of $p$. Williams observed in \cite{Williams(2000)} that the 
parametrized $A$-theory characteristic of \cite{Dwyer-Weiss-Williams(2003)} is actually the image of $\chi(p)$ 
under a \emph{coassembly map}. 

\subsection{The coassembly map} In order to define this coassembly map, we recall first some facts about homotopy limits of categories. 
Let $\cat$ denote the (2-)category of small categories. For every small category $\calI$, the category $\cat^\calI$ of 
$\calI$-shaped diagrams in $\cat$ is enriched over $\cat$ as follows: if $\mathcal{F}, \mathcal{G} \colon \calI\to\cat$ are two functors, then the natural transformations from $\mathcal{F}$ to $\mathcal{G}$ are the objects of a small category 
$\hom(\mathcal{F},\mathcal{G})$. The set of morphisms between two natural transformations 
$\eta,\theta \colon \mathcal{F}\to \mathcal{G}$ is given by
\[\hom(\mathcal{F},\mathcal{G})(\eta,\theta)=\{H\colon \mathcal{F}\times \underline{[1]}\to \mathcal{G}; H_0= \eta, H_1= \theta \}\]
where $\underline{[1]}$ denotes the constant $\calI$-diagram at the category $[1]$. 

\begin{defn}
Let $\calI$ be a small category and $\mathcal{G} \colon \calI\to\cat$ an $\calI$-shaped diagram of small categories. 
The \emph{homotopy limit} of $\mathcal{G}$ is the category
\[\holim \mathcal{G}:=\hom(\calI/?, \mathcal{G})\]
where $\calI/?\colon \calI\to\cat$ is defined on objects by sending $i \in \mathrm{ob}\calI$ to the over category $\calI/i$.
\end{defn}

\begin{rem} 
The nerve of the homotopy limit of an $\calI$-shaped diagram of small categories agrees with the homotopy
 limit of the associated $\calI$-shaped diagram of the nerves as defined in \cite{Bousfield-Kan}. However, this definition should not be confused with the notion of homotopy limit as the derived functor of limit on the category of $\calI$-shaped categories. 
\end{rem}

\begin{rem} 
If the functor $\mathcal{G}$ actually takes values in Waldhausen categories (and exact functors), then, by the naturality of the construction, there is a simplicial category 
$[n] \mapsto \holim wS_n \mathcal{G}$.
\end{rem}

The following lemma is a straightforward exercise in the definition of the homotopy limit.

\begin{lem}\label{technical_lem}
A functor $F\colon\calC\to\holim \mathcal{G}$ determines and is determined by the following data:
\begin{enumerate}
\item for each $i\in \calI$, a functor $F_i\colon \calC\to \mathcal{G}(i)$, and
\item for each morphism $u \colon i\to j$ in $\calI$, a natural transformation $u^!$ from $\mathcal{G}(u) \circ F_i$ to $F_j$, such that $\id_i^!=\id$, and the following cocycle condition is satisfied: for every $v \colon j\to k$ in $\calI$, we have
\[(v \circ u)^!=v^!\circ \mathcal{G}(v) (u^!)\]
as natural transformations between functors $\calC\to \mathcal{G}(k)$.
\end{enumerate}
\end{lem}

We can now define the coassembly map associated to a fibration $p: E \to B$. We assume that $B$ is the geometric realization of a simplicial set $B_\bullet$. Let $\simp(B)$ denote the category of simplices of $B$: an object is a simplicial map $\sigma\colon \Delta^n_\bullet\to B_\bullet$, and a morphism from $\sigma$ to $\tau\colon \Delta^k_\bullet\to B_\bullet$ is a simplicial map $\Delta^n_\bullet\to \Delta^k_\bullet$ making the obvious diagram commutative. We will normally avoid the distinction between the simplex $\sigma$ and its geometric realization. Consider the functor
\[w\Rfd(E\vert_?)\colon \simp(B)\to\cat, \sigma\mapsto w\Rfd(E\vert_\sigma),\]
which is defined on the morphisms by the push-forward maps. For every $\sigma\in\simp(B)$, there is a restriction functor
\[F_\sigma\colon w\Rfd(p)\to w\Rfd(\sigma^*p)\hookrightarrow w\Rfd(E\vert_\sigma)\]
which sends a retractive space $X$ over $E$, which fibers over $B$, to its restriction over the simplex $\sigma$ viewed as a retractive space over the corresponding restriction of $E$. If $u \colon \sigma\to\tau$ is a morphism in $\simp(B)$, then there is a natural transformation induced by the canonical inclusions,
\[u^!\colon u_* F_\sigma\to F_\tau.\]
An easy check shows that the cocycle condition is satisfied. The same construction works when $\Rfd$ is replaced by $S_n\Rfd$, the $n$-th simplicial degree in Waldhausen's $S_\bullet$-construction. Thus, by the Lemma \ref{technical_lem}, we obtain (simplicial) functors
\[c \colon w\Rfd(p)\to\holim_{\simp (B)} w\Rfd(E\vert_?), \quad 
c \colon wS_\bullet \Rfd(p)\to \holim_{\simp (B)} wS_\bullet \Rfd(E\vert_?).\]

\begin{rem}
Again there is a technical point to consider. As it stands, the category $\Rfd(\sigma^* p)$ is \emph{not} a subcategory
of $\Rfd(E\vert_\sigma)$ since an object in the former category is a subset of 
$E\vert_\sigma\amalg (\Delta^n\times \calU)$ while an object in the latter category is a subset of 
$E\vert_\sigma\amalg \calU$. To obtain a functor $\Rfd(\sigma^* P) \to \Rfd(E\vert_\sigma)$, choose 
\begin{itemize}
\item a set-theoretic embedding of the standard simplex $\Delta^n$ into $\calU$, and 
\item a bijection $\calU\times\calU\to\calU$.
\end{itemize}
Then we have $\Delta^n\times\calU\subset\calU\times\calU\cong \calU$ and we obtain a well-defined functor (which is, moreover, an embedding of categories) $\Rfd(\sigma^* p)\hookrightarrow \Rfd(E\vert_\sigma)$.
\end{rem}

We make the following 

\begin{obs} For every functor $\mathcal{G} \colon \calI\to\cat$, the geometric realization defines a 
map $\vert\cdot\vert\colon \vert\holim \mathcal{G} \vert\to\holim\vert \mathcal{G}\vert$. This map is adjoint to the 
simplicial map
\begin{displaymath}
N_{\bullet} \holim \mathcal{G}\xrightarrow{\vert\cdot\vert} \mathrm{Hom}(\Delta^{\bullet}, \map_{\calI}(\vert\calI/?\vert, \vert\mathcal{G}\vert)) =
\sin_{\bullet}(\holim \vert \mathcal{G}\vert),
\end{displaymath}
using the standard model for $\holim \vert \mathcal{G}\vert$ and where $\sin_{\bullet}(-)$ denotes the simplicial set of singular simplices. If $\mathcal{G}$ takes values in Waldhausen categories, then similarly there is a map $\vert\cdot\vert\colon \vert \holim wS_\bullet \mathcal{G}\vert\to\holim \vert wS_\bullet \mathcal{G}\vert$. Moreover, by taking loop spaces, we obtain the map
\[\rho \colon \Omega \vert\holim wS_\bullet \mathcal{G}\vert\to \holim \Omega\vert wS_\bullet \mathcal{G}\vert=\holim K \circ \mathcal{G}.\]
\end{obs}

\begin{defn}
The \emph{$A$-theory coassembly map} is defined to be the composite map
\begin{displaymath}
\coass_p \colon A(p) \xrightarrow{\Omega \vert c \vert} \Omega\vert\holim_{\simp (B)} wS_\bullet \Rfd(E\vert_?)\vert \xrightarrow{\rho} \holim_{\simp (B)} A(E\vert_?).
\end{displaymath}
\end{defn}

The target of the coassembly map is again natural with respect to the covariant and contravariant operations induced respectively by the push-forward and pull-back maps. If $f\colon E\to V$ is a map between fibrations over $B$, 
then there is a natural transformation $\Rfd(E\vert_?)\to \Rfd(V \vert_?)$ inducing
\[f_*\colon \holim_{\simp(B)} A(E\vert_?)\to \holim_{\simp(B)} A(V \vert_?).\]
On the other hand, consider a pullback diagram 
\begin{equation*}
\xymatrix{
E' \ar[d]^{p'} \ar[rr] && E\ar[d]^p\\
B' \ar[rr]^g && B
}
\end{equation*}
and suppose that $g\colon B'\to B$ is the geometric realization of a simplicial map $g_\bullet$. So there is a functor $\simp(g) \colon\simp(B') \to \simp(B)$ and for every object $\sigma$ of $\simp(B')$, there is a canonical isomorphism $E'\vert_\sigma \cong E\vert_{g \circ \sigma}$, since both spaces are just the pullback of $E$ along $g \circ  \sigma$. Hence we obtain a natural isomorphism of functors $$\simp(g)^* A(E\vert_?) \cong A(E'\vert_?)$$ defined 
on $\simp(B')$. Then we can define the pull-back operation as 
\[g^*\colon \holim_{\simp(B)} A(E\vert_?) \to \holim_{\simp(B')}  \simp(g)^* A(E\vert_?) \xrightarrow{\cong} \holim_{\simp(B')} A(E'\vert_?),\]
where the first map is induced by base-change along the functor $\simp(g)$. An easy check shows that $(g\circ h)^*= h^*\circ g^*$. The following proposition, which will be important later on, 
is now obvious.   

\begin{prop}\label{lem:naturality_of_coassembly}
The $A$-theory coassembly map is natural with respect  to the covariant and the contravariant operations.
\end{prop}

\subsection{The $A$-theory characteristic} \label{DWW-RR} We now recall the definition of the 
parametrized $A$-theory Euler characteristic from \cite{Dwyer-Weiss-Williams(2003)}, 
\cite{Williams(2000)}.

\begin{defn}
Let $p \colon E\to B = \vert B_\bullet \vert$ be a fibration with homotopy finite fibers. 
\begin{enumerate}
\item The \emph{bivariant $A$-theory characteristic} $\chi(p)\in A(p)$ is the point determined by the retractive 
space $E\times S^0$ over $E$, considered as an object of $\Rfd(p)$.
\item The \emph{parametrized $A$-theory Euler characteristic} $\chi^{DWW}(p)$ is the image of the 
bivariant $A$-theory characteristic  under the coassembly map 
\[ \coass_p \colon A(p)\to \holim_{\simp (B)} A(E\vert_?). \]
\end{enumerate}
\end{defn}

The element $\chi^{DWW}(p)$ is commonly viewed as a ``classifying map'' from $B$ in the following way 
(see also \cite[I.1.6]{Dwyer-Weiss-Williams(2003)}). There is a canonical weak equivalence from the homotopy limit 
$$\holim_{\simp (B)} A(E\vert_?) = \map_{\simp(B)} (|\simp(B)/?|, A(E\vert_?))$$
to the space of maps over $B$
$$\map_{B}(\hocolim_{\simp (B)} |\simp(B)/?|, \hocolim_{\simp (B)} A(E\vert_?))$$
which is defined by $f \mapsto \hocolim(f)$. Since the canonical map 
$$\hocolim_{\simp(B)} |\simp(B)/?| \to \vert \simp(B) \vert \to B$$ 
is a weak equivalence, it is possible to identify the latter space with a space of 
sections, and thus view the parametrized $A$-theory Euler characteristic as a section
$$\chi^{DWW}(p): B \to A_B(E) : = \hocolim_{\simp (B)} A(E \vert_?)$$
which is uniquely specified up to a contractible space of choices. 

The smooth Riemann-Roch theorem of \cite{Dwyer-Weiss-Williams(2003)}, which describes the element $\chi^{DWW}(p)$ in the case 
where $p$ is a smooth bundle, will be very relevant to our conclusions in the next section. With the convention above in mind, 
we recall the statement (see \cite[Theorem 8.5]{Dwyer-Weiss-Williams(2003)}) and refer to its source for a complete discussion.

\begin{thm}[Dwyer-Weiss-Williams \cite{Dwyer-Weiss-Williams(2003)}] \label{smooth-RR}
Let $p: E \to B$ be a smooth bundle of compact manifolds (possibly with boundary). Then the parametrized $A$-theory 
Euler characteristic $\chi^{DWW}(p): B \to A_B(E)$ is homotopic over $B$, by a preferred homotopy, to the
composition of the parametrized transfer map $\mathrm{tr}(p): B \to (Q_+)_B(E)$ with the fiberwise unit map 
$\eta_p : (Q_+)_B(E) \to A_B(E)$.
\end{thm}

In particular, if $p: E \to B$ is a smooth bundle of compact $d$-dimensional manifolds, then we have a homotopy commutative 
diagram
\begin{equation} \label{smooth-RR-picture}
\xymatrix{
& (Q_+)_B(E) \ar[d]^{\eta_p} \ar[r] & Q(E_+) \ar[d]^{\eta_E} \ar[r] & Q(BO(d)_+) \ar[d]^{\eta_{BO(d)}} \\
B \ar[r] \ar[ur]^{tr(p)} & A_B(E) \ar[r] & A(E) \ar[r] & A(BO(d))
}
\end{equation}
where the right-hand horizontal maps are induced by the classifying map of the vertical tangent bundle over $E$ and the 
other two horizontal maps are defined by the inclusions of the fibers of $p$ into $E$. The vertical maps come from the 
unit tranformation of functors from $X \mapsto Q(X_+)$ to $A$-theory. We recall that this is defined as the composition of 
$$Q(X_+) \longrightarrow A^{\%}(X):= \Omega^{\infty}(\mathbf{A}(\ast) \wedge X_+),$$
given by the unit map $\Sigma^{\infty} S^0 \to \mathbf{A}(\ast)$ of the ring spectrum $\mathbf{A}(\ast)$, with the assembly 
natural map $A^{\%}(X) \to A(X)$. The composite $B \to Q(E_+)$ is the classical Becker-Gottlieb transfer map (see \cite{BeGo}).

\subsection{A scanning map} \label{scanning_map}

We mention the following alternative description of the coassembly map in the special case of a trivial fibration 
$\pi_B:X \times B \to B$. This will be needed in the next section. To simplify the notation, let us abbreviate 
\[W(X, B) : = \vert w \Rfd\fibr{X \times B}{B} \vert.\]
Assume that $B$ is the geometric realization of a simplicial set $B_\bullet$. Pulling back along an $n$-simplex of
$B_{\bullet}$ defines a map 
$$W(X, B) \times \mathrm{Hom}(\Delta^n_{\bullet}, B_{\bullet}) \to W(X, \Delta^n)$$
which is natural in $n$. Thus, for every $x\in W(X,B)$, pulling back along the inclusion of all 
simplices defines a simplicial map 
$x^\ast: B_\bullet\to W(X, \Delta^{\bullet})$. Define the \emph{scanning map} to be the map
\[\scan(X,B): W(X, B) \to \map(B,\left\vert W(X, \Delta^{\bullet}) \right\vert )\]
which sends $x$ to the geometric realization of the simplicial map $x^\ast$. The same construction at the level of $A$-theory yields a 
map
\[\scan(X,B)\colon A\fibr{X\times B}{B}\to \map(B, A_\Delta(X))\]
and the following diagram is commutative, where the vertical maps are given by ``group completion'' \footnote{The term 
``group completion'' here and elsewhere refers to the canonical map $\vert w \mathcal{C} \vert \to K(\mathcal{C})$ for 
every Waldhausen category $\mathcal{C}$, see \cite[1.3, 1.8]{Waldhausen(1985)}.}, 
\[\xymatrix{
W(X,B) \ar[rr]^(.4){\scan(X,B)} \ar[d]
&&  {\map(B, \vert W(X, \Delta^\bullet)\vert)} \ar[d]
\\
A\fibr{X\times B}{B} \ar[rr]^{\scan(X,B)}
&& {\map(B, A_\Delta(X))}
}\]

The comparison of the coassembly and scanning maps will need the following proposition. 

\begin{prop} \label{coass_is_eq}
The $A$-theory coassembly map of $p: E \to B$ is a homotopy equivalence if $B$ is contractible.
\end{prop}
\begin{proof}
This is obvious if $B$ is a point, since then the coassembly map is essentially the identity map. Suppose that $B$ is contractible. Let $F$ be the fiber of $p: E \to B$ over a $0$-simplex of $B$. By naturality, we have a commutative diagram
\[\xymatrix{
A\fibr{E}{B} \ar[rr] \ar[d]^(.55){\simeq} && {\holim_{\simp(B)} A(E\vert_?)} \ar[d]\\
A(F) \ar@{=}[rr] && A(F)  
}\]
where the vertical maps are given by restriction at the $0$-simplex and the horizontal ones by the coassembly map. By the homotopy invariance of Proposition \ref{lem:homotopy_invariance_of_bivariant_stuff}, the left-hand vertical arrow is a homotopy equivalence. Since the functor $A(E\vert_?)$ sends all morphisms to homotopy equivalences, its homotopy limit is homotopy equivalent to the space of sections of a fibration over $\vert\simp (B)\vert$. Under this identification, the right-hand vertical map corresponds to the evaluation of a section at the chosen base-point. Since $\vert\simp (B)\vert\simeq *$, this evaluation map is also a homotopy equivalence and therefore the result follows.
\end{proof}

The next lemma shows that, up to the identification of a homotopy limit with a mapping space of sections, the coassembly and scanning maps of a trivial fibration agree.

\begin{lem} \label{lem:technical_lemma}
There is a commutative diagram in the homotopy category,
\[
\xymatrix{
A\fibr{X\times B}{B} \ar[rr]^{\scan(X,B)} \ar[rd]_{\coass_{\pi_B}} & & \map(B,   A_\Delta(X)) \ar[dl]^{h}_{\cong}\\
& {\underset{\simp (B)}{\holim} A(X \times ?)}. &
}
\]
\end{lem}

\begin{proof}
For convenience, we work here with the thick realization of simplicial spaces which always preserves homotopy equivalences (see \cite{Segal(1974)}). By Proposition \ref{lem:naturality_of_coassembly} the coassembly map is natural. It follows that the coassembly maps for the fibrations $X\times \Delta^n\to \Delta^n$, for varying $n$, fit together to define a simplicial map
\[\coass\colon A_\Delta(X)\to  \bigl\vert [n]\mapsto \holim_{\simp(\Delta^n)} A(X\times ?)\bigr\vert.\]
On the other hand there is a natural pairing
\[\holim_{\simp(B)} A(X\times ?) \times \map(\Delta^n_\bullet, B_\bullet)\to \holim_{\simp(\Delta^n)} A(X\times ?)\]
given by pull-back. It induces a scanning map
\[\scan\colon \holim_{\simp(B)} A(X\times ?) \to \map\bigl(B, \bigl\vert [n]\mapsto \holim_{\simp(\Delta^n)} A(X\times ?)\bigr\vert\bigr).\]
It is a consequence of naturality of both the scanning and the coassembly maps that the following diagram is commutative:
\[\xymatrix{
A\fibr{X\times B}{B}  \ar[rr]^{\scan(X,B)} \ar[d]^(.6){\coass_{\pi_B}} && {\map(B,A_\Delta(X))} \ar[d]^\coass_\simeq \\
{\holim_{\simp (B)} A(X\times ?)} \ar[rr]^(0.4){\scan}_(.4)\simeq  && {\map\bigl(B, \bigl\vert [n]\mapsto \holim_{\simp(\Delta^n)} A(X\times ?)\bigr\vert\bigr)}
}\]
We claim that the labelled arrows are homotopy equivalences, from which the conclusion follows with $h=\scan\inv\circ\coass$.

In fact the right-hand vertical map is induced by a degree-wise homotopy equivalence, as shown in Proposition \ref{coass_is_eq}, and therefore it is a homotopy equivalence. For the lower horizontal map, note that there is a 
chain of homotopy equivalences
\[\holim_{\simp(B)} A(X\times ?) \xrightarrow{\simeq} \holim_{\simp(B)} A(X) \xrightarrow{\cong} \map(\vert\simp(B)\vert, A(X)) \xleftarrow{\simeq}\map(B, A(X)).\]
Here the first map is induced by the projection $X\times ?\to X$, which is a homotopy equivalence. The second map is the standard homeomorphism for the Bousfield--Kan model 
\[\holim_\mathcal{C} F = \map_\mathcal{C}(\vert \mathcal{C}/?\vert, F)\]
of the homotopy limit. The third map is the homotopy equivalence induced by restriction along the last vertex map $\vert\simp(B)\vert\to \vert B\vert$ followed by the projection $\vert B\vert\to B$.

This chain of homotopy equivalences is natural in $B$. So letting $B$ vary over $\{\Delta^n : n \geq 0\}$, we obtain a chain of homotopy equivalences 
\[\bigl\vert [n]\mapsto \holim_{\simp(\Delta^n)} A(X\times ?)\bigr\vert \simeq \bigl\vert [n]\mapsto \map(\Delta^n, A(X))\bigr\vert=\vert \sintop A(X)\vert,\]
the geometric realization of the topological singular construction on the space $A(X)$.

By naturality, the scanning map of the lower line of the diagram extends to all the spaces appearing in the chain. Hence
that map is a homotopy equivalence if and only the corresponding map
\begin{equation} \label{sin_top}
\map(B, A(X)) \to \map(B, \vert\sintop A(X)\vert),
\end{equation}
which is also induced by scanning, is a homotopy equivalence. 
This map is certainly split-injective as the canonical ``co-unit'' map $\vert\sintop A(X)\vert\to A(X)$ induces a left-inverse. But this
canonical map also splits the inclusion of 0-simplices:
\[A(X) = \map(\Delta^0, A(X))\to \vert\sintop A(X)\vert,\]
which is a homotopy equivalence. Thus the co-unit map is also a homotopy equivalence,  hence the same is true for the
map \eqref{sin_top}.
\end{proof}

\section{The B\"{o}kstedt-Madsen map to $A$-theory}

B\"okstedt and Madsen \cite{Boekstedt-Madsen(2011)} defined an infinite loop map 
$$\tau: \Omega B \cob \to A(BO(d)).$$
Broadly speaking, the map sends an $n$-tuple of composable $d$-dimensional cobordisms to 
the union of the cobordisms, regarded as a filtered space, together with the map to 
$BO(d)$ that classifies the tangent bundle (cf.~\cite{Tillmann(1999)}). To make this precise, 
they described the map as a simplicial map on the singular set of $N_{\bullet} \cob$ to the thick model 
for the $A$-theory of $BO(d)$. 

\subsection{Definition of the map $\tilde{\tau}$} Following \cite{Boekstedt-Madsen(2011)}, we define similarly a map
\[
 \tilde{\tau}: \Omega B \cobd \to A(BO(d))
\]
that extends $\tau$ along the map induced by the inclusion functor $\cob \hookrightarrow \cobd$. The map $\tilde{\tau}$ is 
defined by first defining a bisimplicial map between bisimplicial categories
$$\tilde{\tau}_{p,q}: \sin_p N_q \cobdn \to wS_q \Rfd (\Gr_d(\RR^{d+n}),\Delta^{p})$$
and then letting $n \to \infty$ and taking the loop spaces of the geometric realizations. We recall that 
$\sin_{\bullet}(-)$ denotes the simplicial set of singular simplices and the set $\sin_{p} N_{q} 
\cobdn$ is regarded as a category with only identity morphisms. 

A (smooth) $p$-simplex of $N_q \cobdn$ 
$$\sigma: \Delta^p \to \cobdn((M_0, a_0), (M_1, a_1)) \times  \cdots \times \cobdn((M_{q-1},a_{q-1}), (M_q, a_{q}))$$
determines a (smoothly embedded) smooth fiber bundle over $\Delta^p$:
\[
\xymatrix{
E[a_0,a_q] \ar@{^{(}->}[r] \ar[d]^{\pi} & [a_0, a_q] \times \RR_+ \times \RR^{d-2+n} \times \Delta^p \ar[dl] \\
\Delta^p
}
\]
together with a filtering by a sequence of codimension zero smooth sub-bundles over $\Delta^p$, 
$$E[a_0,a_1] \subseteq \cdots \subseteq E[a_0, a_q]$$
where 
$$E[a_0,a_i] =  E[a_0, a_q] \cap ([a_0, a_i] \times \RR^{d-1+n} \times \Delta^p).$$ 
The classifying map of the vertical tangent bundle of $\pi$ restricts to maps
$$\mathrm{tan}^v(\pi): E[a_0, a_i] \to \Gr_d(\RR^{d+n})$$ 
for every $i=1,\dots, q$. This produces a filtered sequence of 
retractive spaces over $\Gr_d(\RR^{d+n}) \times \Delta^p$ whose terms are given by
$$\Gr_d(\RR^{d+n}) \times \Delta^p \rightarrowtail E[a_0,a_i] \cup_{E(a_0)} \Gr_d(\RR^{d+n}) 
\times \Delta^p \stackrel{r}{\rightarrow} \Gr_d(\RR^{d+n}) \times \Delta^p$$ where
$$E(a_0) = E[a_0, a_q] \cap (\{a_0\} \times \RR^{d-1+n} \times \Delta^p)$$ 
fibers also over $\Delta^p$, and the retraction map on $E[a_0,a_i]$ is defined as follows
$$r_{E[a_0,a_i]} = (\mathrm{tan}^v(\pi), \pi).$$
More generally, for $0 \leq i < j \leq q$, let  
$$E[a_i,a_j] =  E[a_0, a_q] \cap ([a_i, a_j] \times \RR^{d-1+n} \times \Delta^p)$$ 
$$E(a_j) =  E[a_0, a_q] \cap (\{a_j\} \times \RR^{d-1+n} \times \Delta^p).$$ 
The collection of the retractive spaces above extends canonically to an object 
$$\{E_{ij}\}_{0 \leq i \leq j \leq q} \in \ob(S_q \Rfd(\Gr_d(\RR^{d+n}), \Delta^{p}))$$ 
where
$$E_{ij} = E[a_i,a_j] \cup_{E(a_i)} \Gr_d(\RR^{d+n}) \times \Delta^p$$
are objects of $\Rfd(\Gr_d(\RR^{d+n}), \Delta^p)$. 

The following lemma is immediate from the definitions.

\begin{lem}
For every $1 \leq n \leq \infty$, the maps $\{\tilde{\tau}_{p,q}\}_{p,q}$ define a 
bisimplicial map $$\tilde{\tau}_{\bullet, \ast}: \sin_{\bullet} N_{\ast} \cobdn \to wS_{\ast} \Rfd (\Gr_d(\RR^{d+n}), \Delta^{\bullet}).$$
\end{lem}

Setting $n = \infty$ and taking the loop spaces of the geometric realizations of these bisimplicial objects, we obtain a (weak 
\footnote{A weak
map of spaces is a zigzag of maps where the wrong-way arrows are weak homotopy equivalences. A weak map from $X$ to $Y$ defines a 
$0$-simplex in the simplicial set of maps from $X$ to $Y$ in the Dwyer-Kan hammock localization of the category of spaces and also
a morphism of the classical localization of the category of spaces at the class of weak homotopy equivalences.}) map:
$$\tilde{\tau}: \Omega B \cobd \stackrel{\sim}{\leftarrow} \Omega |\sin_{\bullet} N_{\bullet} \cobd| \stackrel{\tilde{\tau}}{\rightarrow} A_{\Delta}(BO(d)) \stackrel{\sim}{\leftarrow} A(BO(d)).$$ 
Note that $\tilde \tau$ is a map of loop spaces by definition. We note that the map $\tilde{\tau}$ is defined in exactly the same
way as the map $\tau : \Omega B \cob \to A(BO(d))$ in \cite{Boekstedt-Madsen(2011)}. In particular, the following proposition is 
obvious.

\begin{prop}
The following diagram of (weak) maps commutes in the homotopy category of spaces,
\[
\xymatrix{
\Omega B \cob \ar[rd]_{\tau} \ar@{^{(}->}[rr] && \Omega B \cobd \ar[dl]^{\tilde{\tau}} \\
& A(BO(d)) &
}
\]
\end{prop}

In view of Theorem \ref{Genauer}, it follows that the map $\tau$ factors up to homotopy through 
$Q(BO(d)_+):=\Omega^{\infty}\Sigma^{\infty}BO(d)_+$. Our final goal is to show (Theorem \ref{thm:main_result_2})
that the map $\tilde{\tau}$ can be identified up to homotopy with the canonical unit map 
$\eta_{BO(d)}: Q(BO(d)_+) \to A(BO(d))$. 

\begin{rem}
Similarly we can define maps from other $d$-dimensional cobordism categories with corners to $A(BO(d))$ that 
in turn extend the map $\tilde{\tau}$ above. We refer the reader to \cite[Definition 4.1]{Genauer(2008)} 
for the precise definition of these cobordism categories, and to \cite[Proposition 6.1]{Genauer(2008)} for the general 
result determining their homotopy types in the unoriented case. 
\end{rem}

\subsection{Comparison with the $A$-theory characteristic.} Let $M$ be a compact smooth $d$-dimensional manifold, possibly with boundary, 
neatly embedded in $(0,1) \times \RR_+ \times \RR^{\infty}$. We recall from section 2 that this can be viewed as an 
endomorphism of the empty manifold in $\cobd$ and that there is an inclusion map 
$$i_M: B_{\infty}(M) \hookrightarrow \cobd ((\varnothing,0), (\varnothing,1)) \to  \Omega_{\varnothing} B \cobd.$$ 
Let $\chi^{BM}_M$ denote the restriction of the map $\tilde{\tau}$ along $i_M$, i.e. $$\chi^{BM}_M: = \tilde{\tau} \circ i_M.$$  
Our first goal is to compare the map $\chi^{BM}_M$ with the universal parametrized $A$-theory Euler characteristic 
for $M$-bundles. Explicitly, the map $\chi^{BM}_M$ is defined as follows. 
Write $B_M =\vert \sin_\bullet B_{\infty}(M)\vert$ and let $p_M \colon E_M \to B_M$ be the universal smooth $M$-bundle
pulled back from the tautological bundle over $B_{\infty}(M)$ by the canonical weak equivalence $B_M \stackrel{\sim}{\to} 
B_{\infty}(M)$. 
The vertical tangent bundle defines a map over $B_M$,
\[\mathrm{Tan}^{v}(p_M) \colon E_M \to B_M \times BO(d)\]
which induces a functor
\[ \mathrm{Tan}^{v}(p_M)_* \colon w \Rfd \fibr{E_M}{B_M} \to w\Rfd \fibr{BO(d) \times B_M}{B_M}.\]

The retractive space $E_M \times S^0$ determines a point in $\vert w \Rfd (p_M) \vert$. Note that after ``group completion'', this point becomes the bivariant $A$-theory characteristic of 
$p_M$. The scanning construction applied to the image of this specific point under $\mathrm{Tan}^{v}(p_M)_*$, followed by 
``group completion'', define the map:
\[\chi^{BM}_M \colon B_{\infty}(M)\simeq B_M \to \vert W(BO(d), \Delta^{\bullet}) \vert \to A_{\Delta}(BO(d)) \simeq A(BO(d)).\]
As scanning is compatible with ``group completion'', the map $\chi^{BM}_M$ of B\"okstedt-Madsen agrees 
up to homotopy with the image of $\mathrm{Tan}^v(p_M)_* (\chi(p_M))$ under the scanning construction
\[A\fibr{BO(d)\times B_M}{B_M}\to \map(B_M, A_\Delta(BO(d))),\]
once we have identified $A_\Delta(BO(d))$ with $A(BO(d))$ and $B_\infty(M)$ with $B_M$.

On the other hand, we obtain a new map by passing to the parametrized $A$-theory Euler characteristic 
of $p_M$ first, via the coassembly map, and then applying 
$\mathrm{Tan}^v(p_M)_*$ to it. This is the image of the parametrized
$A$-theory Euler characteristic of $p_M$ under the composite map
\begin{multline*}
\holim_{\simp(B_M)} A (E_M \vert_?) \xrightarrow{\mathrm{Tan}^v(p_M)_*} \holim_{\simp(B_M)} A (BO(d) \times ?) \stackrel{h}{\simeq}
\map(B_M, A_{\Delta}(BO(d))) \\ \simeq \map(B_{\infty}(M), A(BO(d)))
\end{multline*}
or, in other words, the composite map
\begin{multline*} \chi^{DWW}_M \colon B_{\infty}(M)\simeq B_M \xrightarrow{\chi^{DWW}(p_M)} A_{B_M}(E_M) \xrightarrow{\mathrm{Tan}^v(p_M)_*} \hocolim_{\simp(B_M)} A (BO(d) \times ?) 
\\ \simeq A(BO(d)) \times B_M
\end{multline*}
regarded as a section of the trivial fibration. 

\begin{thm}\label{thm:main_result_1}
The maps $\chi^{BM}_M$ and $\chi^{DWW}_M$ agree up to homotopy, i.e. the following diagram
of (weak) maps commutes in the homotopy category of spaces,
\[
 \xymatrix{
&& \Omega B \cobd \ar[d]^{\tilde{\tau}} \\
B_{\infty}(M) \ar[rr]^{\chi^{DWW}_M} \ar[rru]^{i_M} && A(BO(d)). 
}
\]

\end{thm}
\begin{proof}
Let $\tilde{\chi}$ denote the image of $\chi(p_M)$ under the push-forward of $\mathrm{Tan}^v(p_M)$, 
\[ \mathrm{Tan}^v(p_M)_* \colon A\fibr{E_M}{B_M} \to A\fibr{BO(d) \times B_M}{B_M}.\]
By Proposition \ref{lem:naturality_of_coassembly}, the coassembly map commutes with the push-forward map $\mathrm{Tan}^v(p_M)_*$. Hence the image of the parametrized $A$-theory characteristic of $p_M$, under the push-forward map
\[ \holim \mathrm{Tan}^v(p_M)_* \colon \holim_{\simp(B_M)} A (E\vert_?) \to \holim_{\simp(B_M)} A (BO(d) \times ?) 
\]
agrees with the image of $\tilde{\chi}$, under the coassembly map. By definition, this point defines the homotopy class of 
$\chi^{DWW}_M$ via the homotopy equivalence $h$. 
On the other hand, the homotopy class of $\chi^{BM}_M$ is the component of the image of $\tilde{\chi}$ under 
the scanning map. Thus we have the following diagram
\[
\xymatrix{
\chi(p_M) \ar@{|->}[rr]^{\mathrm{Tan}^v(p_M)_*} \ar@{|->}[d]^{\coass} && \tilde{\chi} \ar@{|->}[r]^{\scan} \ar@{|->}[d]^{\coass} & [\chi^{BM}_M] \\
\coass(\chi(p_M)) \ar@{|->}[rr]^(.55){\holim \mathrm{Tan}^v(p_M)_*} && \coass(\tilde{\chi}) \ar@{<->}[r]^{h} & [\chi^{DWW}_M] \ar@{=}[u] &  
}
\]
and the agreement of the two homotopy classes of maps, regarded as elements of $\pi_0 \map(B_{\infty}(M), A(BO(d)))$, follows from the commutative diagram of Lemma \ref{lem:technical_lemma}. 
\end{proof}

\begin{rem} \label{additivity}
Here is an informal interpretation of Theorem \ref{thm:main_result_1} that we will not attempt to make rigorous. 
According to this, the last theorem says that the map $\chi^{DWW}_M$ satisfies \emph{additivity} in $M$ in some strong structured 
sense. Consider morphisms in $\cobd$: $W_1$ from $M_0$ to $M_1$, $W_2$ from $M_1$ to $M_2$ and let $W = W_1 \cup_{M_1} W_2$
be the composition. The additivity property expresses \emph{up to homotopy} the $A$-theory characteristic of a $W$-bundle that
admits a splitting into a $W_1$-bundle and a $W_2$-bundle attached along a $M_1$-bundle as the (loop) sum of the $A$-theory
characteristics of the $W_1$- and $W_2$-bundles minus the $A$-theory characteristic of the $M_1$-bundle. For the additivity of 
the parametrized $A$-theory Euler characteristic in this sense, see \cite{Dorabiala(2002)}. In view of Theorem
\ref{thm:main_result_1}, it suffices to give \emph{a choice of such a homotopy} relating the maps 
$i_{W_2 \circ W_1}$, $i_{W_1}$,$i_{W_2}$ and $i_{M_1}$ mapping into the path space of the cobordism category. But, in 
fact, a \emph{canonical such choice} exists simply by the definition of the cobordism category: every pair of 
composable cobordisms defines a $2$-simplex in $N_{\bullet} \cobd$ and therefore there is a canonical homotopy from the 
path represented by the composition of the cobordisms to the composition of the paths represented by the two cobordisms. This
holds more generally for arbitrary strings of composable cobordisms. Finally, it is worth noting that the thick model for $A$-theory 
allows us to include all these coherent choices of homotopies without changing the homotopy type. 
\end{rem}

\subsection{Comparison with the unit map.} The weak equivalence of Theorem \ref{Genauer} implies that $\Omega B \cobd$ admits the structure of an infinite loop space, i.e. it
is weakly equivalent to the 0-th space of an $\Omega$-spectrum. Broadly speaking, this is the same structure as the one induced by 
the operation of making two embedded cobordisms disjoint and taking their disjoint union. However, some careful analysis is required 
to make this operation precise since there is no canonical choice of making two embedded cobordisms disjoint, in a symmetric manner. A possible approach is to construct a $\Gamma$-space consisting of $n$-tuples of cobordisms that are disjoint. Another one would be to follow the methods of \cite{Boekstedt-Madsen(2011)} to construct deloopings of 
$B \cobd$ geometrically. For our purposes here, we will regard $\Omega B \cobd$ as an infinite loop space with the
structure that is induced by $Q(BO(d)_+)$. 

We recall that the space of configurations of finite sets of points in $\RR^n$ labelled by elements of a space $X$ 
$$\coprod_{m \geq 0} \Emb(\{1, \cdots, m\}, \RR^n) \times_{\Sigma_m} X^m$$
can be adjusted up to weak equivalence into a topological monoid whose group completion is weakly equivalent to 
$\Omega^{n}\Sigma^n(X_+)$, see \cite{Segal(1973)}. Such a model is the topological monoid $C_n(X)$ whose elements 
are triples $(S, \xi, t)$ where:
\begin{itemize}
\item[(i)] $t \in [0, \infty)$ and $S \subseteq (0,t) \times \RR^{n-1}$ is a finite subset,
\item[(ii)] $\xi : S \to X$ is a map that defines the labels.
\end{itemize}
This is regarded as a subspace of 
$$\coprod_{ m \geq 0} \RR_{\geq 0} \times (\Emb(\{1, \cdots, m\}, \RR^{n}) \times_{\Sigma_m} X^m)$$
with the subspace topology. This space becomes an associative topological monoid under the operation
$$(S, \xi, t) \cdot (S', \xi', t') \colon = (S \cup T_t(S'), \xi \cup \xi', t + t')$$
where $T_t: (0,t') \times \RR^{n-1} \to (t, t+t') \times \RR^{n-1}$ is the translation by $t$ and 
$\xi \cup \xi': S \cup T_t(S') \to X$ is defined by $\xi$ and $\xi'$ in the obvious way. Letting $n \to \infty$, we define 
$$C_{\infty}(X) \colon = \colim_n C_n(X).$$

By well known results in the theory of infinite loop spaces (\cite{Boardman-Vogt}, \cite{May}, \cite{Segal(1974)}), 
an identification of $\Emb(\{1, \cdots, n\}, \RR^{\infty})$ as a model for $E\Sigma_n$ shows that the group completion 
of the topological monoid $C_{\infty}(X)$ admits infinite deloopings  
and, moreover, that it is weakly equivalent to $Q(X_+)$. Thus we may regard $\Omega B (C_{\infty}(X))$ as the $0$-th
term of an $\Omega$-spectrum. 

For later purposes, we will need such an explicit identification. This allows a comparison between $C_{\infty}(X)$
and the weakly equivalent topological monoid from subsection \ref{unit_map_model}:
$$ |w \Rdd(X, \Delta^\bullet) | = \coprod_{m \geq 0} E \Sigma_m \times_{\Sigma_m} |\sin_{\bullet}(X)|^m.$$

\begin{lem}\label{lem:explicit_identification_of_models_for_stable_homotopy}
(i) There is a natural weak equivalence $$\beta_X \colon |\sin_{\bullet} C_{\infty}(X)| \stackrel{\sim}{\to}
\vert w \Rdd(X, \Delta^\bullet)\vert$$
where $\beta_X$ is a map of topological monoids. Moreover, the map $\beta_X$ induces a weak equivalence between 
the classifying spaces.

(ii) The composite map
\[
\vert \sing_\bullet X \vert\to \vert\sing_\bullet C_\infty(X)\vert\xrightarrow{\beta_X} \vert w\Rdd(X, \Delta^\bullet)\vert 
\to \overline{Q}_{\Delta}(X)
\]
is up to homotopy the adjoint to the stable map $\theta_X$ from Proposition \ref{thick_model_stable_htpy}. (Here the first map in the 
composition is induced by the inclusion $X\to C_1(X)$ which sends $x$ to the configuration of one particle with label $x$, 
sitting at $\frac12\in (0,1)$.)
\end{lem}
\begin{proof}
(i) The map is defined by a simplicial map, denoted also by
$$\beta_X : \sin_{\bullet}C_n(X) \to w \Rdd(X, \Delta^\bullet),$$
and letting $n \to \infty$. An $p$-simplex of $C_n(X)$ defines a bundle as follows 
\[
\xymatrix{
E \ar@{^{(}->}[r] \ar[d]^{\pi} & (0,\infty) \times \RR^{n-1} \times X \times \Delta^p \ar[dl] \\
\Delta^p
}
\]
whose fibers are discrete spaces. Forgetting about the ambient Euclidean space, we obtain an object 
of $\Rd(X, \Delta^p)$:
\[
\xymatrix{
E \sqcup (X \times \Delta^p) \ar@<1ex>[r] \ar[d]^{\pi} & X \times \Delta^p \ar[dl] \ar@<1ex>[l] \\
\Delta^p
}
\]
This defines an object of $\Rdd(X, \Delta^p)$ by taking its image under an equivalence $\Rd(X, \Delta^p) \to \Rdd(X, \Delta^p)$. The correspondence clearly defines a simplicial map. Note 
that the simplicial set $\sin_{\bullet} C_{\infty}(X)$ is a simplicial monoid where the multiplication is defined 
pointwise by the multiplication in $C_{\infty}(X)$. The identity of $\sin_p C_{\infty}(X)$ is the constant map at the 
unit element of $C_{\infty}(X)$ which is defined by the empty subset $S$ with $t=0$. The map $\beta_X$ sends this unit element to 
the zero object of $w \Rdd(X, \Delta^p)$. Furthermore, $w \Rdd(X, \Delta^\bullet)$ is a simplicial monoidal category where the 
monoidal product is defined levelwise by the coproduct functor in $w \Rdd(X, \Delta^n)$ for all $n \geq 0$. Then it is easy to see that the product of two 
$n$-simplices is sent to the coproduct of their values under the simplicial map $\beta_X$. Finally, we note that the map 
$\beta_X$ is induced by $\Sigma_m$-equivariant simplicial maps, for all $m \geq 0$,
$$\sin_{\bullet} \Emb(\{1, \cdots, m\}, (0, \infty) \times \RR^{n-1}) \times \sin_{\bullet}(X^m) \to E\Sigma_m \times (\sin_{\bullet}(X)^m)$$
which is clearly a weak equivalence. (Here $E\Sigma_m$ denotes the nerve of the transport category of $\Sigma_m$,
and not its classifying space.) It follows that $\beta_X$ is a weak equivalence, as required. Then the last claim also follows 
immediately because both monoids are well-pointed.

(ii) This is immediate from the definition of $\theta_X$ in Proposition \ref{thick_model_stable_htpy}.
\end{proof}

Let $\mathcal{C}_0(X)$ be the $0$-dimensional cobordism category with background space $X$ as a tangential structure in the 
sense of \cite[Section 5]{GMTW(2009)}. (Tangential structures are also briefly discussed in subsection \ref{tangential_str}.)
We recall that we work with the model of ``discrete cuts'' as explained in section 2 (see \cite[Remark 2.1(ii)]{GMTW(2009)}). 
Note that the topological monoid $C_{\infty}(X)$ is exactly the reduced version of the $0$-dimensional cobordism category, in 
the sense of \cite[Remark 2(i)]{GMTW(2009)}, with background space $X$ (but without ``discrete cuts''). Translation of 
configurations along the auxiliary coordinate defines a functor 
$$\mathcal{C}_0(X) \longrightarrow C_{\infty}(X)$$
which induces a weak equivalence between the classifying spaces, see \cite[Remark 4.5]{GMTW(2009)}, \cite{Boekstedt-Madsen(2011)}. 

Following the discussion in \cite[\S 3]{Segal(1973)}, the monoid $C_n(X)$ (and similarly the category $\mathcal{C}_0(X)$) 
can be further adjusted in order to obtain a nice description of the group completion map to $\Omega^n \Sigma^n (X_+)$. This adjustment amounts to 
making choices of tubular neighborhoods of the embedded finite sets of points $S \subseteq \RR^n$. Let $\widetilde{C}_n(X)$ be the 
space whose elements are triples 
$(S, \xi, t)$ where
\begin{itemize}
\item[(i)] $t \in [0, \infty)$ and $S \subseteq (0,t) \times \RR^{n-1}$ is a subspace of finitely many disjoint open 
unit $n$-disks,
\item[(ii)] $\xi : S \to X$ is a locally constant map that defines the labels.
\end{itemize}
This space is also an associative topological monoid under an operation defined similarly as above. Restricting to the orgins 
of the embedded $n$-disks defines an inclusion map
$$\iota\colon \widetilde{C}_n(X) \hookrightarrow C_n(X)$$
and it is easy to see that this subspace is a deformation retract of $C_n(X)$. Then there is a collapse map 
$$\widetilde{C}_n(X) \longrightarrow \Omega^n \Sigma^n X_+$$
which induces a weak equivalence between the classifying spaces, see \cite[\S 3]{Segal(1973)}. Letting $n \to \infty$, we 
define $$\widetilde{C}_{\infty}(X) \colon = \colim_n \widetilde{C}_n(X)$$
and Segal \cite{Segal(1973)} shows that the induced map
\[\widetilde C_\infty(X) \to Q(X_+)\]
is a group completion, i.e., it induces a weak equivalence
\begin{equation}\label{eq:segals_map}
 \mathcal{S} \colon \Omega B(\widetilde C_\infty(X)) \to Q(X_+).  
\end{equation}

\begin{cor}\label{cor:comparing_segal_with_BM}
The map $\Omega^\infty(\theta_X)$ from Proposition \ref{thick_model_stable_htpy}, as a map 
in the homotopy category of spaces, is given by the following zigzag of weak equivalences:
\[
Q(X_+)
\xleftarrow{\mathcal{S}} 
\Omega B\vert \sing_\bullet \widetilde C_\infty(X)\vert  
\xrightarrow{\Omega B(\iota)} 
\Omega B\vert \sing_\bullet C_\infty(X) \vert 
\xrightarrow{\Omega B(\beta_X)} 
\overline{Q}_{\Delta}(X).
\]
\end{cor}

\begin{proof}
By part (ii) of Lemma \ref{lem:explicit_identification_of_models_for_stable_homotopy}, the adjoint of $\theta_X$ factors through the inclusion $\vert \sing_\bullet X\vert \to \vert \sing_\bullet C_\infty(X)\vert$,
which one may lift to $\vert \sing_\bullet \tilde C_\infty(X)\vert $. But 
the square
\[\xymatrix{
\vert \sing_\bullet X\vert \ar[r] \ar[d] & X \ar[d]\\
\vert \sing_\bullet \widetilde C_\infty(X) \vert \ar[r] & Q(X_+)
}\]
commutes up to homotopy. This shows that the composite of the inclusion $X\to Q(X_+)$ with the zigzag of the 
statement is adjoint to the stable map $\theta_X$. This implies the claim as all the maps in the 
zigzag are maps of infinite loop spaces.
\end{proof}

Similarly to the definition of $\widetilde{C}_{\infty}(X)$, we can define a variant $\widetilde{\mathcal{C}}_0(X)$ of the cobordism 
category $\mathcal{C}_0(X)$ by letting the configurations have a unit disk as a tubular neighborhood. There is an analogous 
inclusion of categories $\widetilde{\mathcal{C}}_0(X) \hookrightarrow \mathcal{C}_0(X)$ which induces a weak equivalence on objects 
and on morphism spaces. Moreover, the obvious diagram of functors commutes,
\[
\xymatrix{
\widetilde{\mathcal{C}}_0(X) \ar[r] \ar[d] & \mathcal{C}_0(X) \ar[d] \\
\widetilde{C}_{\infty}(X) \ar[r] & C_{\infty}(X). 
} 
\]

Let $D^d$ denote the $d$-dimensional closed disk and $D^d_m$ a disjoint union of $m$ copies of $D^d$. There is a functor
$$\psi: \widetilde{\mathcal{C}}_{0}(\Gr_d(\RR^{\infty})) \longrightarrow \cobd,$$
which, roughly speaking, sends a configuration of $m$ points in $\RR^{\infty}$ labelled by $d$-dimensional linear 
subspaces to the associated configuration of $m$ disjoint linearly embedded $d$-disks in $\RR^\infty$. More precisely, 
it is defined on objects by $(\varnothing, a) \mapsto (\varnothing, a)$. A non-identity morphism 
$(S \subseteq (a, b) \times \RR^{n - 1}, \xi: S \to \Gr_d(\RR^n))$, where $S$ is finite collection of disjoint unit 
$n$-disks and $\xi$ a locally constant map, defines a finite collection of disjoint linearly embedded 
$d$-disks in $(a, b) \times \RR^{n-1}$ by intersecting, for every $n$-disk component $S_i \subseteq S$,
\begin{itemize}
\item[(i)] the smaller closed $n$-disk $S'_i \subseteq S_i$ of radius $\frac12$, with
\item[(ii)] the $d$-plane through the origin of $S_i$ defined by the label at this point.
\end{itemize}
This defines a finite collection of $d$-disks of radius $\frac12$ embedded in $(a,b) \times \RR^{n-1}$. By adding a 
new ambient coordinate, we can fix a canonical way of embedding each of these linearly embedded $d$-disks to
a neatly smoothly embedded $d$-disk in $(a, b) \times \RR_+ \times \RR^{n-1}$. Then the new collection of 
embedded $d$-disks is a morphism in $\cobd$ which we define to be the image of $\psi$ at $(S,\xi)$. 

\begin{lem} \label{group_comp}
The functor $\psi$ induces a weak equivalence between the classifying spaces.
\end{lem}
\begin{proof} The description of the weak equivalence $\tilde{\alpha}$ in section 2 is essentially a 
generalization of the collapse map $\mathcal{S}$ to embedded manifolds of higher dimension. There is a canonical 
path joining the image of an element $(S, \xi, 1) \in \widetilde{C}_{\infty}(\Gr_d(\RR^{\infty}))$ under the collapse map 
$\mathcal{S}$, to the image of the element $\psi(S, \xi)$ under the map $B_{\infty}(D^d_m) \to \Omega B \cobd \to Q(BO(d)_+)$ as 
described in section 2, where $(S, \xi)$ is regarded as a morphism from $(\varnothing, 0)$
to $(\varnothing, 1)$ with $|\pi_0(S)|= m$ and $\psi(S, \xi)$ comes with a choice of a tubular neighborhood by definition. This can 
be used to define a homotopy from the composition 
$$\Omega B (\widetilde{\mathcal{C}}_0(\Gr_d(\RR^{\infty})) \stackrel{\Omega B (\psi)}{\longrightarrow} \Omega B \cobd \stackrel{\tilde{\alpha}}{\to} Q(BO(d)_+)$$
to the composition 
$$\Omega B (\widetilde{\mathcal{C}}_0(\Gr_d(\RR^{\infty})) \stackrel{\sim}{\longrightarrow} \Omega B \widetilde{C}_{\infty}(\Gr_d(\RR^{\infty})) \stackrel{\mathcal{S}}{\to} Q(BO(d)_+),$$
which proves the claim.
\end{proof}

Denote by $$\eta = \eta_{BO(d)} \colon Q(BO(d)_+) \to A(BO(d))$$ 
the unit transformation of $A$-theory evaluated at $BO(d)$. We also let 
$BO(d) = \Gr_d(\RR^\infty)$ in order to simplify the notation in the following proof.

\begin{thm}\label{thm:main_result_2} 
The map $\tilde{\tau}$ can be identified, by a preferred weak equivalence, with the unit map, i.e. 
the following diagram of (weak) maps commutes in the homotopy category of spaces,
\[
\xymatrix{
\Omega B \cobd \ar[rr]^(.45){\tilde\alpha}_(.45){\sim} \ar[rd]_{\tilde{\tau}} && Q(BO(d)_+) \ar[dl]^{\eta} \\
& A(BO(d)). & \\
}
\]
\end{thm}
\begin{proof} 
First note that we may precompose with the weak equivalence $\Omega B(\psi)$ of Lemma \ref{group_comp}. 
As we showed in the proof of that lemma, the composite map
$\tilde\alpha\circ \Omega B(\psi)$ agrees up to homotopy with the map $\mathcal{S}$ from \eqref{eq:segals_map}. 

Then the following diagram of bisimplicial categories shows two maps from a bisimplicial set to a bisimplicial category, 
\[\xymatrix{\ar @{} [drr] |{\Longleftarrow}
\sin_{\bullet} N_{\ast} \widetilde{\mathcal{C}}_0(BO(d)) \ar[r] \ar[d]^{\psi} & \sin_{\bullet} N_{\ast} \widetilde C_\infty(BO(d)) \ar[r]^\iota &
\sin_{\bullet} N_{\ast} C_{\infty}(BO(d)) \ar[d]^{\beta_{BO(d)}} \\
{\sing_\bullet} N_{\ast} \cobd \ar[r]^(.45){\tilde{\tau}} & w S_{\ast} \Rfd(BO(d), \Delta^\bullet) & wS_{\ast}\Rdd(BO(d), \Delta^\bullet) \ar[l] 
}\]
which agree up to a natural tranformation, which is given by including to a bundle of $d$-disks the subbundle of 
points defined by restricting to the origins of the $d$-disks fiberwise. This natural transformation shows 
that the two compositions induce homotopic maps after passing to the geometric realizations.

This shows that the map $\tilde{\tau} \circ \Omega B (\psi)$, as a map in the homotopy category of spaces, 
agrees with the lower composition in the following diagram of maps in the homotopy category of spaces,
\[\xymatrix{
\Omega B \widetilde{\mathcal{C}}_0(BO(d)) \cong {\Omega B}\widetilde C_\infty(BO(d))  
\ar[rr]^{\mathcal{S}}_{\cong} \ar[d]^{\Omega B(\beta_{BO(d)}\circ\iota)}_{\cong}
&&
Q(BO(d)_+)
\ar[d]^{\eta_{BO(d)}} 
\\
\overline{Q}_{\Delta}(BO(d))
\ar[r]^{\eta^{\Delta}_{BO(d)}}
&
A_{\Delta}(BO(d))
\ar[r]^\cong
&
A(BO(d)).
}\]
Finally, it remains to show that the last diagram in the homotopy category commutes. 
By Corollary \ref{cor:comparing_segal_with_BM}, the composite map 
$$Q(BO(d)_+) \to \overline{Q}_{\Delta}(BO(d)),$$ 
going through the left-hand corner of the diagram, agrees with the map $\Omega^\infty(\theta_{BO(d)})$. Then the result 
follows from Proposition \ref{unit_map} where we used this last map to identify the unit map with 
$\eta^{\Delta}_{BO(d)}$.
\end{proof}

Using geometric methods to construct deloopings of $B \cob$, it was shown in \cite{Boekstedt-Madsen(2011)} that the map $\tau$ 
is an infinite loop map. The same result for the map $\tilde{\tau}$ is now a consequence of Theorem \ref{thm:main_result_2}.

\begin{cor} \label{infinite-loop-map}
$\tilde{\tau}$ is a map of infinite loop spaces.
\end{cor}

\begin{rem}
In view of Theorem \ref{thm:main_result_1} and Remark \ref{additivity}, Theorem \ref{thm:main_result_2} can be seen as expressing
a structured form of an additivity property for the factorization of $\chi^{DWW}_M$ through the unit map. The 
combination of the two theorems implies the homotopy commutativity of the outer triangle in Diagram \eqref{smooth-RR-picture} 
of subsection \ref{DWW-RR}:
\[
\xymatrix{
     && Q(BO(d)_+) \ar[d]^\eta\\
B_M \ar[rru] \ar[rr]_{\chi^{DWW}_M}  && A(BO(d))
}\]
for every smooth $d$-manifold $M$ (possibly with boundary). 
\end{rem}

\section{Concluding Remarks}

\subsection{Tangential structures} \label{tangential_str} Similar ideas apply to the case of cobordism categories with tangential structures. 
Let $\theta: X \to \Gr_d(\RR^{\infty})$ be a fibration. The authors of \cite{GMTW(2009)} defined a cobordism category 
$\cob^{\theta}$ of manifolds equipped with a tangential $\theta$-structure, i.e. a lift of the stable tangent bundle to
$X$. The main theorem of \cite{GMTW(2009)} identifies the homotopy type of $B \cob^{\theta}$ with the infinite loop space 
$\Omega^{\infty -1} \mathrm{MT}\theta$ of the Thom spectrum associated with the stable bundle $\theta^*(- \gamma_d)$. Genauer 
\cite{Genauer(2008)} considered the cobordism category $\cobd^{\theta}$ of $d$-dimensional manifolds with boundary and a tangential 
$\theta$-structure and showed that there is a weak equivalence
$$ \tilde{\alpha}^{\theta} : \Omega B \cobd^{\theta} \stackrel{\sim}{\to} Q(X_+).$$
(The main theorem of \cite[Theorem 4.5]{Genauer(2008)} identifies the homotopy type of a cobordism category with corners with the 
infinite loop space associated to a homotopy colimit of Thom spectra. The actual identification of this spectrum with the suspension 
spectrum of the space $X$ is similar to \cite[Proposition 3.1]{GMTW(2009)}, see also \cite[Section 5]{GMTW(2009)}.) 

By replacing the vertical tangent bundle map to $\Gr_d(\RR^{\infty})$ with the $\theta$-structure to $X$, we can similarly 
define a (weak) map  
\[ \tilde{\tau}^{\theta}: \Omega B \cobd^{\theta} \to A(X). \]
Let $M$ be a compact smooth $d$-dimensional manifold. Following the notation of \cite[Section 5]{GMTW(2009)},  let 
$$B_{\infty}^{\theta}(M) = \Emb^{\theta}(M, [0,1] \times \RR^{\infty})/ \Diff(M)$$
where $\Emb^{\theta}(M, [0,1] \times \RR^{\infty})$ denotes the space of (neat) embeddings of $M$ together with 
compatible choices of a $\theta$-structure. The proof of the following $\theta$-version of Theorem 
\ref{thm:main_result_1} is essentially the same.

\begin{thm} \label{thm:main_result_1-theta}
The following diagram of (weak) maps commutes in the homotopy category of spaces,
\[
\xymatrix{
&& \Omega B \cobd^{\theta} \ar[d]^{\tilde{\tau}^{\theta}} \\
B^{\theta}_{\infty}(M) \ar[rr]^{\chi^{DWW}_M} \ar[rru]^{i^{\theta}_M} && A(X). 
}
\]
\end{thm}

Furthermore, following similar arguments, we obtain the $\theta$-versions of Lemma \ref{group_comp} and
Theorem \ref{thm:main_result_2}. 

\begin{lem} \label{group_comp_2}
There is a functor $\psi^{\theta}: \widetilde{\mathcal{C}}_{0}(X) \to \cobd^{\theta}$, defined 
similarly to $\psi$ of Lemma \ref{group_comp}, which induces a weak equivalence between the classifying 
spaces.
\end{lem}

\begin{thm} 
The map $\tilde{\tau}^{\theta}$ can be identified, by a preferred weak equivalence, with the unit map, i.e. the following diagram of 
(weak) maps commutes in the homotopy category of spaces,
\[
\xymatrix{
\Omega B \cobd^{\theta} \ar[rr]^(.45){\tilde\alpha^{\theta}}_(.45){\sim} \ar[rd]_{\tilde{\tau}^{\theta}} && Q(X_+) \ar[dl]^{\eta_X} \\
& A(X). & \\
}
\]
\end{thm}

We also have the following immediate consequence (cf. Corollary \ref{infinite-loop-map}).

\begin{cor}
$\tilde{\tau}^{\theta}$ is a map of infinite loop spaces.
\end{cor}

Finally, we mention two cases of special interest. First, consider the oriented cobordism category $\cobd^+$ defined 
by $\theta$ being 
the orientation cover. In this case, there is a homotopy commutative diagram as follows,
\[
\xymatrix{
\Omega B \cobd^+ \ar[r]^(.4){\sim} &  Q(BSO(d)_+) \ar[r]^(.55){\eta} & A(BSO(d)) \\
\Omega B \cob^+ \ar@{^{(}->}[u] \ar[rru]_{\tau^+} 
}
\] 
The weak equivalence in the diagram is shown in \cite[Proposition 6.2]{Genauer(2008)}.

Second, consider the cobordism category $\cobd(X)$ where $X$ denotes a background space. This is the category associated to
the trivial fibration $\theta: \Gr_d(\RR^{\infty}) \times X \to \Gr_d(\RR^{\infty})$. The correspondence 
$$X \mapsto \pi_* (\Omega B \cobd(X)),$$
viewed as a functor in $X$, is the (unreduced) homology of $X$ with respect to the suspension spectrum $\Sigma^{\infty} BO(d)_+$. 
In this case, we have a homotopy commutative diagram as follows,
\[
\xymatrix{
\Omega B \cobd(X) \ar[dr]^{\tilde{\tau}^X} \ar[rr]^{\sim} && Q((BO(d) \times X)_+) \ar[dl]^{\eta} \\
& A(BO(d) \times X). &
}
\]
We note that since $\tilde{\tau}^X$ is a natural transformation of spectra from an excisive functor, it is determined 
up to homotopy by its canonical factorization through the excisive approximation to the functor 
$X \mapsto \mathbf{A}(BO(d) \times X)$, i.e. $$X \mapsto \Omega^{\infty}(\mathbf{A}(BO(d)) \wedge X_+).$$
(See \cite[8.1-8.3]{Dwyer-Weiss-Williams(2003)}.) The latter factorization is a natural transformation of excisive functors 
and thus it is determined by the map of spectra $\tilde{\tau}$, which has been identified with the unit map at $BO(d)$.

\subsection{A splitting of the cobordism category}
 
A version of the B\"okstedt-Madsen map in the oriented $2$-dimensional case was defined in  \cite{Tillmann(1999)}. This map was used there to deduce the existence of a certain splitting of the homotopy type 
of that cobordism category.  The arguments apply similarly in higher dimensions. Let $M$ be a closed $d$-dimensional manifold embedded in $\RR^{\infty}$, so that it may be regarded as a (endo)morphism in $\cob$. Thus it defines a point in $\Omega B \cob$ and, using the infinite 
loop space structure, we can extend the inclusion of this point to an infinite loop map 
$$j_M \colon QS^0 \to \Omega B \cob.$$
By composing $j_M$ with the composite infinite loop map $$ \Omega B \cob \stackrel{\tau}{\to} A(BO(d)) \stackrel{e_{\ast}}{\to} A(\ast) \stackrel{Tr}{\to} Q S^0 $$
where $e: BO(d) \to \ast$ and $Tr$ denotes Waldhausen's trace map \cite{Waldhausen(1979)}, we obtain a self map of $Q S^0$. By Theorem \ref{thm:main_result_1}, it is easy to see that the homotopy class of this map can be identified with the Euler characteristic of $M$, $\chi(M) \in \mathbb{Z} \cong \pi_0^s. $ Thus, for every such $M$, we obtain a geometric description of a splitting of a copy of the localized sphere spectrum $(QS^0)[\chi(M)^{-1}]$ from 
$\Omega B \cob$, as infinite loop spaces. 

These splittings can also be realized at the level of the Thom spectrum $\mto(d)$ as follows. The bordism class of $M$ defines an element $[M] \in \pi_0 \mto(d)$ represented by a map $QS^0 \to \Omega^{\infty} \mto(d)$. Up to the weak equivalence $\alpha$ of Theorem \ref{GMTW}, this is the same map as $j_M$. Composition with the map $\Omega^{\infty} \mto(d) \to Q(BO(d)_+)$, given by 
the addition of the tautological bundle, and the map $Q(BO(d)_+) \to QS^0$, which collapses $BO(d)$ to a point, 
produces the same self-map of $QS^0$, specified as multiplication by $\chi(M)$. If $M \subseteq \RR^N$, this is represented by the composite 
\[ S^N \to \mathrm{Th}(\nu_M) \to \mathrm{Th}(\gamma^{\perp}_{d, N-d}) \to \mathrm{Th}(\gamma^{\perp}_{d, N-d} \oplus \gamma_{d,N-d}) \cong S^N \wedge \Gr_d(\RR^{N})_+ \to S^N \]
where the first map is the Pontryagin-Thom collapse map, the second map is defined by the classifying map for the 
normal bundle of $M$, the third map is the addition of the tautological bundle and the fourth map is given by  collapsing at the basepoint.

\appendix

\section{Products in Bivariant A-theory}

\numberwithin{thm}{section}

We briefly discuss the construction of products in bivariant $A$-theory (see also \cite{Williams(2000)}). For technical reasons, we 
need to consider a slightly modified model for the Waldhausen category $\Rfd(-)$. For any 
fibration $p: E \to B$, let $\fRfd(p)$ be the Waldhausen subcategory of $\Rfd(p)$ spanned by those retractive spaces $(X, i, r)$ 
over $E$ such that the retraction map $r: X \to E$ is a fibration. This full subcategory is closed in $\Rfd(p)$ under pushouts along a 
cofibration, so it becomes a Waldhausen category with the induced structure from $\Rfd(p)$. It is easy to show that the inclusion 
exact functor $\fRfd(p) \hookrightarrow \Rfd(p)$ induces a weak equivalence in $K$-theory. The drawback of this construction is 
that it is covariantly functorial \emph{only} with respect to fibrations. The readers who prefer to think about $\Rfd(p)$ instead, 
could do so, as long as they replace the retraction maps with fibrations throughout the steps of the construction.

Our goal is to show that for any pair of fibrations $f: E \to V$ and $g: V \to B$, there is a natural map 
$$A(f) \wedge A(g) \longrightarrow A(p)$$
where $p = g \circ f$. This can be obtained from a bi-exact functor 
$$\otimes : \fRfd(f) \times \fRfd(g) \to \fRfd(p)$$ 
which is defined as follows. Given objects $(X, i_X, r_X)$ of $\fRfd(f)$ and $(Y, i_Y,r_Y)$ of 
$\fRfd(g)$, we first consider the pullback $(X', i_{X'}, r_{X'}) : =  f^*(Y,i_Y,r_Y)$ as an object of $\calR(E)$. Then we form the 
external smash product $X \wedge_E X'$ of the two retractive spaces over $E$, i.e. the retractive space over $E \times E$ that is
defined by the pushout diagram
\[
\xymatrix{
X \times E \cup_{E \times E} E \times X' \ar[rr] \ar@{>->}[d] && E \times E \ar@{>->}[d] \\
X \times X' \ar[rr] && X \wedge_E X'
}
\]
Note that the induced retraction $r_X \wedge_E r_{X'}: X \wedge_E X' \to E \times E$ is again a fibration. Finally, by taking the
pullback along the diagonal $\Delta : E \to E \times E$, we obtain a retractive space over $E$ which we denote by 
$(X \otimes Y, i_{X \otimes Y}, r_{X \otimes Y})$. This construction is clearly functorial and it preserves cofibrations and weak 
equivalences. Thus, it remains to check that $(X \otimes Y, i_{X \otimes Y}, r_{X \otimes Y})$ is an object of $\fRfd(p)$. 

The induced retraction $r_{X \otimes Y}$ is a fibration, so it suffices to show that the homotopy finiteness condition is satisfied. Note that 
this is a condition for each point of the base space $B$. By restricting attention to the fibers over 
a point of $B$, throughout the construction, we can assume that $B$ is the one-point space. Under this assumption, 
it suffices to show that pair $(X \otimes Y, E)$ is homotopy finite. 

\begin{lem} \label{CW-homotopy-type}
Consider a diagram as follows, 
\[
\xymatrix{
F' \ar[d] \ar@{>->}[rr] && F \ar[d] \\
E' \ar[dr]_p \ar@{>->}[rr] && E \ar[dl]^q \\
& B & 
}
\]
where $p$ and $q$ are fibrations, $F'$ and $F$ denote the fibers at a point $b \in B$, and the horizontal maps are cofibrations. Suppose also that the fiber pair $(F, F')$ is homotopy finite. If $(B, B_0)$ 
is homotopy equivalent to relative (finite) CW-complex, then so is the pair $(E, E_{|B_0} \cup E')$.
\end{lem}
\begin{proof}
We may assume that $(B, B_0)$ is a relative CW complex. If it is relative finite, then it suffices, by induction, to consider
only the case where $B$ is obtained from $B_0$ by attaching a single $n$-cell along some attaching map $f\colon S^{n-1}\to B_0$. 
Then the inclusion $E_{\vert B_0} \cup E' \to E$ may be described as the map
\[E_{\vert B_0} \cup_{E'_{\vert S^{n-1}}} E'_{\vert D^n} \to E_{\vert B_0} \cup_{E_{\vert S^{n-1}}} E_{\vert D^n}\]
induced by the inclusions on the individual components. As $E'\to E$ is a cofibration, there is a commutative diagram of fiberwise maps over $D^n$
\[\xymatrix{
F'\times D^n  \ar[rr]^\phi_\simeq \ar@{>->}[d] && E'_{\vert D^n} \ar@{>->}[d] \\
F\times D^n \ar[rr]^\phi_\simeq && E_{\vert D^n}  
}\]
where the horizontal maps are fiber homotopy equivalences. It induces a commutative square
\[\xymatrix{
E_{\vert B_0} \cup_{F' \times S^{n-1}} F'\times D^n  
       \ar[rr]^\simeq \ar@{>->}[d] 
&& E_{\vert B_0} \cup_{E'_{\vert S^{n-1}}} E'_{\vert D^n} \ar@{>->}[d] 
\\
E_{\vert B_0} \cup_{F \times S^{n-1}} F\times D^n  
       \ar[rr]^\simeq 
&& E_{\vert B_0} \cup_{E_{\vert S^{n-1}}} E_{\vert D^n} 
}\]
The horizontal maps are homotopy equivalences so it is enough to show that the left-hand column is a homotopy finite pair. This 
follows from the assumption that the pair $(F, F')$ is homotopy finite. In fact, assuming that $(F, F')$ is actually a 
finite relative CW-complex, then the left-hand inclusion in the diagram defines also a finite relative CW-complex which
has one $(n+k)$-cell for each $k$-cell in the relative CW-structure of $(F, F')$.

In the general case, where $(B,B_0)$ is not necessarily relative finite, then the pair $(E, E_{|B_0} \cup E')$ is defined by a 
direct colimit of cofibrations which are homotopy equivalent to relative CW-complexes, so it is also homotopy equivalent to a 
relative CW-complex.
\end{proof}

To finish the proof of the construction, we apply the lemma to the following diagram
\[
\xymatrix{
f^{-1}(r_Y(y)) \ar[d] \ar@{>->}[rr] && (fr_X)^{-1}(r_Y(y)) \ar[d] \\
X' = E \times_E X' \ar[dr] \ar@{>->}[rr] && X \times_E X' \ar[dl] \\
& Y & 
}
\]
where the top row shows the fibers at $y \in Y$. By assumption, the pair $(Y, V)$ is homotopy finite. It follows that the 
pair $(X \times_E X', X \cup_E X')$ is also homotopy finite. But note that the latter pair is relative 
homeomorphic to $(X \otimes Y, E)$ and therefore the required homotopy finiteness condition is satisfied.


\bibliographystyle{amsplain}

\begin{thebibliography}{10}

\bibitem{BDW(2009)}
B.~Badzioch, W.~Dorabia{\l}a, and B.~Williams, \emph{Smooth parametrized
  torsion: a manifold approach}, Adv. Math. \textbf{221} (2009), no.~2,
  660--680. 

\bibitem{BeGo}
J.~C. Becker and D.~H. Gottlieb, \emph{Transfer maps for fibrations and
  duality}, Compositio Math. \textbf{33} (1976), no.~2, 107--133.

\bibitem{Boardman-Vogt}
J. ~M. Boardman and R. ~M. Vogt, \emph{Homotopy invariant algebraic structures on topological spaces},
Lecture Notes in Mathematics, Vol. 347, Springer-Verlag, Berlin, 1973.
     
\bibitem{Boekstedt-Madsen(2011)}
M.~B{\"o}kstedt and I.~Madsen, \emph{The cobordism category and {W}aldhausen's
  {K}-theory}, Preprint, arXiv:1102.4155, 2011.

\bibitem{Bousfield-Kan}
A.~K. Bousfield and D.~M. Kan, \emph{Homotopy limits, completions and
  localizations}, Lecture Notes in Mathematics, Vol. 304, Springer-Verlag,
  Berlin, 1972. 

\bibitem{Cisinski(2010)}
D.-C. Cisinski, \emph{Invariance de la {$K$}-th\'eorie par \'equivalences
  d\'eriv\'ees}, J. K-Theory \textbf{6} (2010), no.~3, 505--546.

\bibitem{Dorabiala(2002)}
W. Dorabia{\l}a, \emph{The double coset theorem formula for algebraic {$K$}-theory of spaces},
$K$-Theory \textbf{25} (2002), no. ~3, 251--276.

\bibitem{Dwyer-Weiss-Williams(2003)}
W.~Dwyer, M.~Weiss, and B.~Williams, \emph{A parametrized index theorem for the
  algebraic {$K$}-theory {E}uler class}, Acta Math. \textbf{190} (2003), no.~1,
  1--104. 

\bibitem{Fulton-MacPherson(1981)}
W. Fulton, R. MacPherson, \emph{Categorical framework for the study of singular spaces},
Mem. Amer. Math. Soc. \textbf{31} (1981), no. ~243.

\bibitem{GMTW(2009)}
S.~Galatius, U.~Tillmann, I.~Madsen, and M.~Weiss, \emph{The homotopy type of
  the cobordism category}, Acta Math. \textbf{202} (2009), no.~2, 195--239.


\bibitem{Genauer(2008)}
J.~Genauer, \emph{Cobordism categories of manifolds with corners}, 
Trans. Amer. Math. Soc.  \textbf{364} (2012),  no. ~1, 519--550. 

\bibitem{Kieboom}
R.~W. Kieboom, \emph{A pullback theorem for cofibrations}, Manuscripta Math.
  \textbf{58} (1987), no.~3, 381--384. 

\bibitem{May}
J. ~P. May, \emph{The geometry of iterated loop spaces}, Lectures Notes in Mathematics, Vol. 271, 
Springer-Verlag, Berlin, 1972.

\bibitem{Segal(1973)}
G.~Segal, \emph{Configuration-spaces and iterated loop-spaces}, Invent. Math.
  \textbf{21} (1973), 213--221.

\bibitem{Segal(1974)}
\bysame, \emph{Categories and cohomology theories}, Topology \textbf{13}
  (1974), 293--312.

\bibitem{Tillmann(1999)}
U.~Tillmann, \emph{A splitting for the stable mapping class group}, Math. Proc.
  Cambridge Philos. Soc. \textbf{127} (1999), no.~1, 55--65.

\bibitem{Waldhausen(1979)}
F.~Waldhausen, \emph{Algebraic {$K$}-theory of topological spaces. {II}},
  Algebraic topology, {A}arhus 1978 ({P}roc. {S}ympos., {U}niv. {A}arhus,
  {A}arhus, 1978), Lecture Notes in Math., vol. 763, Springer, Berlin, 1979,
  pp.~356--394. 

\bibitem{Waldhausen(1982)}
\bysame, \emph{Algebraic {$K$}-theory of spaces, a manifold approach}, Current
  trends in algebraic topology, {P}art 1 ({L}ondon, {O}nt., 1981), CMS Conf.
  Proc., vol.~2, Amer. Math. Soc., Providence, R.I., 1982, pp.~141--184.

\bibitem{Waldhausen(1985)}
\bysame, \emph{Algebraic {$K$}-theory of spaces}, Algebraic and geometric
  topology ({N}ew {B}runswick, {N}.{J}., 1983), Lecture Notes in Math., vol.
  1126, Springer, Berlin, 1985, pp.~318--419.

\bibitem{WeWi(1995)}
M. ~Weiss, B. ~ Williams \emph{Assembly}, Novikov conjectures, index theorems and 
rigidity, Vol.\ 2, (Oberwolfach, 1993), London Math. Soc. Lecture Note Ser.,
no. 227 (1995), Cambridge Univ. Press, Cambridge, pp. 332--352.

  
\bibitem{Williams(2000)}
B.~Williams, \emph{Bivariant {R}iemann {R}och theorems}, Geometry and topology:
  {A}arhus (1998), Contemp. Math., vol. 258, Amer. Math. Soc., Providence, RI,
  2000, pp.~377--393.

\end{thebibliography}
\providecommand{\MR}{\relax\ifhmode\unskip\space\fi MR }
\providecommand{\MRhref}[2]{%
  \href{http://www.ams.org/mathscinet-getitem?mr=#1}{#2}
}
\providecommand{\href}[2]{#2}

\end{document}